%% file: handbook.tex
\documentclass[krantz1,ChapterTOCs]{krantz} 
\usepackage{fixltx2e,fix-cm}
\usepackage{amssymb,amsmath,amsthm}
\usepackage{graphicx}
\usepackage{makeidx}
\usepackage{multicol}

\usepackage{color}
\usepackage{enumerate,amssymb}
\usepackage[arrow,curve,matrix,tips,2cell]{xy}
  \SelectTips{eu}{10} \UseTips
  \UseAllTwocells

\DeclareMathAlphabet\EuR{U}{eur}{m}{n}
\SetMathAlphabet\EuR{bold}{U}{eur}{b}{n}

\theoremstyle{plain}
\newtheorem{theorem}{Theorem}[section]
\newtheorem{lemma}[theorem]{Lemma}
\newtheorem{proposition}[theorem]{Proposition}
\newtheorem{corollary}[theorem]{Corollary}

\theoremstyle{definition}

\newtheorem{definition}[theorem]{Definition}
\newtheorem{example}[theorem]{Example}

\newtheorem{remark}[theorem]{Remark}

{\catcode`@=11\global\let\c@equation=\c@theorem}
\renewcommand{\theequation}{\thetheorem}

\newcommand{\comsquare}[8]                   
{\begin{CD}
#1 @>#2>> #3\\
@V{#4}VV @V{#5}VV\\
#6 @>#7>> #8
\end{CD}
}

\newcommand{\xycomsquare}[8]                   
{\xymatrix
{#1 \ar[r]^{#2} \ar[d]^{#4} &
#3 \ar[d]^{#5}  \\
#6\ar[r]^{#7} &
#8
}
}

\newcommand{\xycomsquareminus}[8]                      
{\xymatrix{#1 \ar[r]^-{#2} \ar[d]^-{#4} &
#3 \ar[d]^-{#5}  \\
#6\ar[r]^-{#7} &
#8
}
}

\newcommand{\aut}{\operatorname{aut}}

\newcommand{\colim}{\operatorname{colim}}

\newcommand{\Ext}{\operatorname{Ext}}

\newcommand{\Sq}{\operatorname{Sq}}

\newcommand{\higherlim}[3]{{\setbox1=\hbox{\rm lim}
        \setbox2=\hbox to \wd1{\leftarrowfill} \ht2=0pt \dp2=-1pt
        \mathop{\vtop{\baselineskip=5pt\box1\box2}}
        _{#1}}^{#2}#3}

\newcommand{\version}[1]                       
{\begin{center} last edited on #1\\
last compiled on \today\\
name of texfile: \jobname
\end{center}
}

\newcounter{commentcounter}

\newcommand{\abstract}[1]{
\begin{center}
\begin{minipage}{11.8cm}
ABSTRACT. #1
\end{minipage}	
\end{center}
}

\def \mmod{/\mkern-3mu /}

\newcommand{\nd}{\not\!|}
\newcommand{\mc}[1]{\mathcal{#1}}
\newcommand{\mb}[1]{\mathbb{#1}}
\newcommand{\mr}[1]{\mathrm{#1}}

\newcommand{\abs}[1]{\lvert #1 \rvert}

\newcommand{\bra}[1]{\langle #1 \rangle}
\newcommand{\br}[1]{\overline{#1}}

\newcommand{\td}[1]{\widetilde{#1}}

\newcommand{\ZZ}{\mathbb{Z}}

\newcommand{\CC}{\mathbb{C}}
\newcommand{\QQ}{\mathbb{Q}}

\newcommand{\FF}{\mathbb{F}}
\newcommand{\GG}{\mathbb{G}}
\newcommand{\MS}{\mathbb{S}}
\newcommand{\AF}{\mathbb{A}}
\newcommand{\PP}{\mathbb{P}}
\newcommand{\TAF}{\mathrm{TAF}}
\newcommand{\TMF}{\mathrm{TMF}}
\newcommand{\Tmf}{\mathrm{Tmf}}
\newcommand{\tmf}{\mathrm{tmf}}
\newcommand{\bo}{\mathrm{bo}}
\newcommand{\MU}{\mathrm{MU}}
\newcommand{\MUP}{\mathrm{MUP}}

\DeclareMathOperator*{\holim}{holim}

\begin{document}

\frontmatter


\setcounter{page}{1} 

\mainmatter

\include{ch1}


\printindex

\end{document}

%% file: ch1.tex

\renewcommand{\leftmark}{Mark Behrens}

\chapterauthoronly{Mark Behrens}{}

\chapter{Topological modular and automorphic forms}

\renewcommand{\theequation}{\thesection.\arabic{equation}}

\section{Introduction}

The spectrum of topological modular forms (TMF)\index{topological modular forms} was first introduced by Hopkins and Miller \cite{HopkinsMiller}, \cite{HopkinsTMFI}, \cite{HopkinsTMFII}, and Goerss and Hopkins constructed it as an $E_\infty$ ring spectra (see \cite{Behrenstmf}).  Lurie subsequently gave a conceptual approach to $\TMF$ using his theory of spectral algebraic geometry \cite{Lurie}.  Lurie's construction relies on a general theorem \cite{LurieI}, \cite{LurieII}, which was used by the author and Lawson to construct spectra of topological automorphic forms (TAF) \cite{taf}.  

The goal of this article is to give an accessible introduction to the $\TMF$ and $\TAF$ spectra.
Besides the articles mentioned above, there already exist many excellent such surveys  (see \cite{HopkinsMahowald}, \cite{Rezk}, \cite{Lawsonsurvey}, \cite{Goerss}, \cite{Goersstmf}, \cite{tmf}).  Our intention is to give an account which is somewhat complementary to these existing surveys.  We assume the reader knows about the stable homotopy category, and knows some basic algebraic geometry, and attempt to give the reader a concrete understanding of the big picture while deemphasizing many of the technical details.  Hopefully, the reader does not find the inevitable sins of omission to be too grievous.

In Section~\ref{sec:elliptic} we recall the definition of a complex orientable ring spectrum $E$, and its associated formal group law $F_E$.  We then explain how an algebraic group $G$ also gives rise to a formal group law $F_G$, and define elliptic cohomology theories to be complex orientable cohomology theories whose formal group laws arise from elliptic curves.  We explain how a theorem of Landweber proves the existence of certain \emph{Landweber exact} elliptic cohomology theories.

We proceed to define topological modular forms in Section~\ref{sec:tmf}.  We first begin by recalling the definition of classical modular forms as sections of powers of a certain line bundle on the compactification $\br{\mc{M}}_{ell}$ of the moduli stack $\mc{M}_{ell}$ of elliptic curves.  Then we state a theorem of Goerss-Hopkins-Miller, which states that there exists a sheaf of $E_\infty$ ring spectra $\mc{O}^{top}$ on the \'etale site of $\br{\mc{M}}_{ell}$ whose sections over an affine
$$ \mr{spec}(R) \xrightarrow[\mr{etale}]{C} \mc{M}_{ell} $$
recover the Landweber exact elliptic cohomology theory associated to the elliptic curve $C$ it classifies.  The spectrum $\Tmf$ is defined to be the global sections of this sheaf:
$$ \Tmf := \mc{O}^{top}(\br{\mc{M}}_{ell}). $$
We compute $\pi_*\Tmf[1/6]$, and use that to motivate the definition of connective topological modular forms ($\tmf$) as the connective cover of $\Tmf$, and periodic topological modular forms ($\TMF$) as the sections over the non-compactified moduli stack: 
$$ \TMF := \mc{O}^{top}(\mc{M}_{ell}). $$ 

The homotopy groups of $\TMF$ at the primes $2$ and $3$ are more elaborate.  While we do not recount the details of these computations, we do indicate the setup in Section~\ref{sec:pitmf}, and state the results in a form that we hope is compact and understandable.  The computation of $\pi_*\Tmf_{(p)}$ for $p = 2,3$ is then discussed in Section~\ref{sec:piTmf}.  We introduce the notion of the height of a formal group law, and use it to create a computable cover of the compactified moduli stack $(\br{\mc{M}}_{ell})_{\ZZ_{(p)}}$. 
By taking connective covers, we recover the homotopy of groups of $\tmf_{(p)}$.

In Section~\ref{sec:chromatic}, we go big picture.  We explain how complex cobordism associates to certain ring spectra $E$ a stack
$$ \mc{X}_E \rightarrow \mc{M}_{fg} $$
over the moduli stack of formal group laws.  The sheaf $\mc{O}^{top}$ serves as a partial inverse to $\mc{X}_{(-)}$, in the sense that where it is defined,
we have
\begin{align*}
\mc{O}^{top}(\mc{X}_E) & \simeq E, \\
\mc{X}_{\mc{O}^{top}(\mc{U})} & \simeq \mc{U}.
\end{align*} 
We describe the height filtration of the moduli stack of formal groups $\mc{M}_{fg}$, and explain how chromatic localizations of the sphere realize this filtration in topology.  The stacks associated to chromatic localizations of a ring spectrum $E$ are computed by pulling back the height filtration to $\mc{X}_E$.  We then apply this machinery to $\Tmf$ to compute its chromatic localizations, and explain how chromatic fracture is closely connected to the approach to $\pi_*\Tmf$ discussed in Section~\ref{sec:piTmf}.

We then move on to discuss Lurie's theorem, which expands $\mc{O}^{top}$ to the \'etale site of the moduli space of \emph{$p$-divisible groups}.  After recalling the definition, we state Lurie's theorem, and explain how his theorem simultaneously recovers the Goerss-Hopkins-Miller theorem on Morava $E$-theory, and the Goerss-Hopkins-Miller sheaf $\mc{O}^{top}$ on $\mc{M}_{ell}$.  We then discuss a class of moduli stacks of Abelian varieties (\emph{PEL Shimura stacks of type $U(1,n-1)$}) which give rise to spectra of topological automorphic forms \cite{taf}.

There are many topics which should have appeared in this survey, but regrettably do not, such as the Witten orientation, the connection to $2$-dimensional field theories, spectral algebraic geometry, and equivariant elliptic cohomology, to name a few.  We compensate for this deficiency in Section~\ref{sec:further} with a list of such topics, and references to the literature for further reading.

\subsection*{Acknowledgments}

The author would like to thank Haynes Miller, for soliciting this survey, as well as Sanath Devalapurkar, Lennart Meier, John Rognes, Taylor Sutton, and Markus Szymik for valuable suggestions and corrections.  The author was partially supported from a grant from the National Science Foundation.

\section{Elliptic cohomology theories}\label{sec:elliptic}

\subsection*{Complex orientable ring spectra}

Let $E$ be a (homotopy associative, homotopy commutative) ring spectrum.  

\begin{definition}
A \emph{complex orientation} of $E$ is an element
$$ x \in \td{E}^*(\CC P^\infty) $$
such that the restriction
$$ x\vert_{\CC P^1} \in \td{E}^*(\CC P^1) $$
is a generator (as an $E_*$-module).
A ring spectrum which admits a complex orientation is called \emph{complex orientable}\index{complex orientatable}.
\end{definition}

For complex oriented ring spectra $E$, the Atiyah-Hirzebruch spectral sequence collapses to give
\begin{align}
E^*(\CC P^\infty) & = E_*[[x]], \label{eq:ECP}\\
E^*(\CC P^\infty \times \CC P^\infty) & = E_*[[x_1, x_2]]
 \end{align}
where $x_i$ is the pullback of $x$ under the $i$th projection.

Consider the map
$$ \mu: \CC P^\infty \times \CC P^\infty \rightarrow \CC P^\infty $$
which classifies the universal tensor product of line bundles.
Quillen \cite{Quillen} (see also \cite{Adams}) observed that 
because $\mu$ gives $\CC P^\infty$ the structure of a homotopy commutative, homotopy associative $H$-space, the power series
$$ F_E(x_1, x_2) := \mu^*x \in E_*[[x_1, x_2]] $$
is a (commutative, $1$ dimensional) \emph{formal group law}\index{formal group law} over $E_*$, in the sense that it satisfies
\begin{enumerate}
\item $F_E(x, 0) = F_E(0,x) = x,$ 
\item $F_E(x_1, F_E(x_2, x_3)) = F_E(F_E(x_1, x_2), x_3),$
\item $F_E(x_1, x_2) = F_E(x_2, x_1)$.
\end{enumerate}

\begin{example}\label{ex:add} 
Let $E = H\ZZ$, the integral Eilenberg-MacLane spectrum.  Then the complex orientation $x$ is a generator of $H^2(\CC P^\infty)$, and 
$$ F_{H\ZZ}(x_1, x_2) = x_1 + x_2. $$
This is the \emph{additive formal group law} $F_{\rm{add}}$.
\end{example}

\begin{example}\label{ex:mult}
Let $E = KU$, the complex $K$-theory spectrum.  Then the class
$$ x := [L_{\mr{can}}] - 1 \in \td{KU}^0(\CC P^\infty) $$
(where $L_{\mr{can}}$ is the canonical line bundle) gives a complex orientation for $KU$, and
$$ F_{KU}(x_1, x_2) = x_1 + x_2 + x_1x_2. $$
This is the \emph{multiplicative formal group law} $F_{\rm{mult}}$.
\end{example}

Example (\ref{ex:mult}) above is an example of the following.

\begin{definition}
An \emph{even periodic} ring spectrum\index{even periodic ring spectrum} is a ring spectrum $E$ so that
\begin{equation}\label{eq:odd} \pi_{\mr{odd}}E = 0 \end{equation}
and such that $E_{2}$ contains a unit.
\end{definition}

It is easy to see using a collapsing Atiyah-Hirzebruch spectral sequence argument that (\ref{eq:odd}) is enough to guarantee the complex orientability of an even periodic ring spectrum $E$.  The existence of the unit in $E_2$ implies one can take the complex orientation to be a class
$$ x \in \td{E}^0(\CC P^\infty) $$
giving
\begin{equation}\label{eq:E0CP} 
E^0(\CC P^\infty) = E_0[[x]].
\end{equation}
It follows that in the even periodic case, for such choices of complex orientation, we can regard the formal group law $F_E$ as a formal group law over $E_0$, and (\ref{eq:E0CP}) can be regarded as saying
$$ E^0(\CC P^\infty) = \mc{O}_{F_{E}}. $$
where the latter is the ring of functions on the formal group law.  Then it follows that we have a canonical identification
$$ E_2 = \td{E}^0(\CC P^1) = (x)/(x)^2 = T^*_0F_E $$
Here $(x)$ is the ideal generated by $x$ in $E^0(\CC P^\infty)$ and $T^*_0F_E$ is the cotangent space of $F_E$ at $0$.
The even periodicity of $E$ then gives
\begin{equation}\label{eq:E2i}
E_{2i} \xleftarrow{\cong} E_2^{\otimes_{E_0} i} = (T^*_0F_E)^{\otimes i}.
\end{equation}
Here (\ref{eq:E2i}) even makes sense for $i$ negative: since $E_2$ is a free $E_0$-module of rank $1$, it is invertible (in an admittedly trivial manner), and $T^*_0F_E$ is invertible since it is a line bundle over $\mr{spec}(E_0)$.

\subsection*{Formal groups associated to algebraic groups}

Formal group laws also arise in the context of algebraic geometry.  Let $G$ be a $1$-dimensional commutative algebraic group over a commutative ring $R$.  If the line bundle $T_eG$ (over $\mr{spec} (R)$) is trivial, there exists a coordinate $x$ of $G$ at the identity $e \in G$.  We shall call such group schemes \emph{trivializable}\index{trivializable group scheme}.  In this case the group structure
$$ G \times G \rightarrow G $$
can be expressed locally in terms of the coordinate $x$ as a power series
$$ F_G(x_1, x_2) \in R[[x_1, x_2]]. $$
The unitality, associativity, and commutativity of the group structure on $G$ makes $F_G$ a formal group law over $R$.  The formal groups in Examples~\ref{ex:add} and \ref{ex:mult} arise in this manner from the additive and multiplicative groups $\GG_a$ and $\GG_m$ (defined over $\ZZ$) by making appropriate choices of coordinates:
\begin{align*}
F_{\GG_a} & = F_{add}, \\
F_{\GG_m} & = F_{mult}.
\end{align*}

It turns out that if we choose different coordinates/complex orientations, we will still get isomorphic formal group laws.
A \emph{homomorphism}\index{homomorphism, of formal group laws} $f: F \rightarrow F'$ of formal group laws over $R$ is a formal power series
$$ f(x) \in R[[x]] $$
satisfying
$$ f(F(x_1, x_2)) = F'(f(x_1), f(x_2)). $$
If the power series $f(x)$ is invertible (with respect to composition) then we say that it is an isomorphism.
Clearly, choosing a different coordinate on a trivializable commutative $1$-dimensional algebraic group gives an isomorphic formal group law.  One similarly has the following proposition.

\begin{proposition}
Suppose that $x$ and $x'$ are two complex orientations of a complex orientable ring spectrum $E$, with corresponding formal group laws $F$ and $F'$.  Then there is a canonical isomorphism between $F$ and $F'$.
\end{proposition}

\begin{proof}
Using (\ref{eq:ECP}), we deduce that $x' = f(x) \in E_*[[x]]$.  It is a simple matter to use the resulting change of coordinates to verify that $f$ is an isomorphism from $F$ to $F'$. 
\end{proof}

The only $1$-dimensional connected algebraic groups over an algebraically closed field are $\GG_a$, $\GG_m$, and elliptic curves.  As we have shown that there are complex orientable ring spectra which yield the formal groups of the first two, it is reasonable to consider the case of elliptic curves.

\subsection*{The elliptic case}

\begin{definition}[\cite{AHS}]
An \emph{elliptic cohomology theory}\index{elliptic cohomology theory} consists of a triple
$$ (E,C,\alpha) $$
where
\begin{align*}
E & = \text{an even periodic ring spectrum},\\
C & = \text{a trivializable elliptic curve over $E_0$}, \\
(\alpha: F_C \xrightarrow{\cong} F_E) & = \text{an isomorphism of formal group laws}.
\end{align*} 
\end{definition}

\begin{remark}
Note that every elliptic curve $C$ which admits a Weierstrass presentation
$$ C: y^2 + a_1xy + a_3 y = x^3 + a_2 x^2 + a_4 x + a_6 $$
is trivializable, since $z = x/y$ is a coordinate at $e \in C$.  
\end{remark}

For an elliptic cohomology theory $(E, C, \alpha)$,  the map $\alpha_*$ gives an isomorphism
\begin{equation*}
 T_e^*C \cong T^*_0F_C \xleftarrow[\cong]{\alpha^*} T^*_0F_E.
\end{equation*}
It follows that we have a canonical isomorphism
\begin{equation}\label{eq:TC}
E_{2i} \cong (T^*_eC)^{\otimes i}.
\end{equation}

It is reasonable to ask when elliptic cohomology theories exist.  This was first studied by Landweber, Ravenel, and Stong \cite{LRS} using the Landweber Exact Functor Theorem \cite{Landweber}.  Here we state a reformulation of this theorem which appears in \cite{Naumann} (this perspective originates with Franke \cite{Franke} and Hopkins \cite{coctalos}).

\begin{theorem}[Landweber Exact Functor Theorem]\label{thm:LEFT}\index{Landweber exact functor theorem} Suppose that $F$ is a formal group law over $R$ whose classifying map
$$ \mr{spec}(R) \xrightarrow{F} \mc{M}_{fg} $$
to the moduli stack of formal groups is flat.  Then there exists a unique (in the homotopy category of ring spectra) even periodic ring spectrum $E$ with $E_0 = R$ and $F_E \cong F$.  
\end{theorem}

\begin{corollary}\label{cor:Landweber}
Suppose that $C$ is a trivializable elliptic curve over $R$ whose associated formal group law $F_C$ satisfies the hypotheses of the Landweber exact functor theorem.  Then there exists an elliptic cohomology theory $E_C$ associated to the elliptic curve $C$.
\end{corollary}

\begin{remark}
For us, a \emph{stack}\index{stack} is a functor
$$ \mr{Rings}^{op} \rightarrow \mr{Groupoids} $$
which satisfies a descent condition with respect to a given Grothendieck topology.   The moduli stack of formal groups $\mc{M}_{fg}$ associates to a ring $R$ the groupoid whose objects are formal group schemes over $R$ which are Zariski locally (in $\mr{spec}(R)$) isomorphic to the formal affine line $\widehat{\mb{A}}^1$, and whose morphisms are the isomorphisms of such. 
\end{remark}

\begin{remark}
While Landweber's original formulation of the exactness criterion may be less conceptual than that given above, it is much easier to check in practice.
\end{remark}

The problem with Landweber's theorem is that while it gives a functor
$$ \{ \text{Landweber flat formal groups} \} \rightarrow \mr{Ho}(\mr{Spectra}), $$
this functor does not refine to a point-set level functor to spectra.  

\section{Topological modular forms}\label{sec:tmf}

\subsection*{Classical modular forms}

Let $\mc{M}_{ell}$ denote the moduli stack of elliptic curves (over $\mr{spec}(\ZZ)$).\index{moduli stack of elliptic curves}  It is the stack whose $R$-points is the groupoid of elliptic curves over $R$ and isomorphisms.  Consider the line bundle $\omega$ on $\mc{M}_{ell}$ whose fiber over an elliptic curve $C$ is given by the cotangent space at the identity
$$ \omega_C = T^*_eC. $$ 
The moduli stack of elliptic curves $\mc{M}_{ell}$ admits a compactification $\br{\mc{M}}_{ell}$ \cite{DeligneRapoport} where we allow our elliptic curves to degenerate to singular curves in the form of \emph{N\'eron $n$-gons}.  The line bundle $\omega$ extends over this compactification.
The space of (integral) modular forms of weight $k$ is defined to be the global sections (see \cite{Katz})\index{modular form}
\begin{equation}\label{eq:MF}
 MF_k := H^0(\br{\mc{M}}_{ell}, \omega^{\otimes k}).
\end{equation}

The complex points $\mc{M}_{ell}(\CC)$ admit a classical description (see, for example, \cite{SilvermanII}).  Let $\mc{H} \subseteq \CC$ denote the upper half plane.  Then we can associate to a point $\tau \in \mc{H}$ an elliptic curve $C_\tau$ over $\CC$ by defining
$$ C_\tau := \CC/(\ZZ + \ZZ\tau). $$
Every elliptic curve over $\CC$ arises this way.  Let $SL_2(\ZZ)$ act on $\mc{H}$ through M\"obius transformations:
$$ \begin{pmatrix}
a & b \\
c & d
\end{pmatrix}
\cdot \tau = \frac{a\tau + b}{c \tau + d}.
$$
Two such elliptic curves $C_\tau$ and $C_{\tau'}$ are isomorphic if and only if 
$\tau' = A\cdot \tau$ for some $A$ in $SL_2(\ZZ)$.  It follows that 
$$ \mc{M}_{ell}(\CC) = \mc{H} \mmod SL_2(\ZZ). $$
In this language, a modular form $f \in MF_k$ can be regarded as a meromorphic function on $\mc{H}$ which satisfies
\begin{equation*}
 f(\tau) = (c\tau+d)^{-k}f(A\cdot \tau)
 \end{equation*}
for every
$$ A = \begin{pmatrix} a & b \\ c & d \end{pmatrix} \in SL_2(\ZZ). $$
The condition of extending over the compactification $\br{\mc{M}}_{ell}$ can be expressed over $\CC$ by requiring that the Fourier expansion (a.k.a. \emph{$q$-expansion})\index{q-expansion}
$$ f(q) = \sum_{i\in \ZZ} a_i q^i \quad (q := e^{2\pi i \tau})$$
satisfies $a_i = 0$ for $i < 0$ (a.k.a. ``holomorphicity at the cusp'').

\subsection*{The Goerss-Hopkins-Miller sheaf}

The following major result of Goerss-Hopkins-Miller \cite{HopkinsMiller}, \cite{Behrenstmf}
gives a topological lift of the sheaf $\bigoplus_i \omega^{\otimes i}$.

\begin{theorem}[Goerss-Hopkins-Miller]\label{thm:GHM}
There is a homotopy sheaf of $E_\infty$-ring spectra $\mc{O}^{top}$\index{$\mc{O}^{top}$} on the \'etale site of $\br{\mc{M}}_{ell}$ with the property that the spectrum of sections
$$ E_C := \mc{O}^{top}(\mr{spec}(R) \xrightarrow{C} \mc{M}_{ell}) $$
associated to an \'etale map $\mr{spec}(R) \rightarrow \mc{M}_{ell}$ classifying a trivializable elliptic curve $C/R$ is an elliptic cohomology theory for the elliptic curve $C$.
\end{theorem}

\begin{remark}
Since the map
\begin{gather*}
\mc{M}_{ell} \rightarrow \mc{M}_{fg} \\
C \mapsto F_C
\end{gather*}
is flat, it follows that every elliptic cohomology theory $E_C$ coming from the theorem above could also have been constructed using Corollary~\ref{cor:Landweber}.  The novelty in Theorem~\ref{thm:GHM} is: 
\begin{enumerate}
\item the functor $\mc{O}^{top}$ lands in the point-set category of spectra, rather than in the homotopy category of spectra,
\item the spectra $E_C$ are $E_\infty$, not just homotopy ring spectra, and 
\item the functor $\mc{O}^{top}$ can be evaluated on non-affine \'etale maps of stacks
$$ \mc{X} \rightarrow \br{\mc{M}}_{ell}. $$
\end{enumerate}
\end{remark}

Elaborating on point (3) above, the ``homotopy sheaf''\index{homotopy sheaf} property of $\mc{O}^{top}$ implies that for any \'etale cover 
$$ \mc{U} = \{ U_i \} \rightarrow \mc{X} $$
the map
$$ \mc{X} \rightarrow  \mr{Tot} \left( \prod_{i_0} \mc{O}^{top}(U_{i_0}) \Rightarrow \prod_{i_0, i_1} \mc{O}^{top}(U_{i_0} \times_{\mc{X}} U_{i_1}) \Rrightarrow \cdots \right)$$ 
is a equivalence.  The Bousfield-Kan spectral sequence of the totalization takes the form
\begin{equation}\label{eq:BKSS}
 E_1^{s,t} =  \prod_{i_0, \ldots, i_s} \pi_t \mc{O}^{top}(U_{i_0} \times_{\mc{X}} \cdots \times_{\mc{X}} U_{i_s}) \Rightarrow \pi_{t-s}\mc{O}^{top}(\mc{X}).
\end{equation}
Because $\br{\mc{M}}_{ell}$ is a separated Deligne-Mumford stack, there exists a cover of $\mc{X}$ by affines, and all of their iterated pullbacks are also affine.  Since every elliptic curve is locally trivializable over its base, we can refine any such cover to be a cover which classifies trivializable elliptic curves. In this context we find (using (\ref{eq:TC})) that the $E_1$-term above can be identified with the \v Cech complex 
$$ E_1^{s,2k} = \check{C}^s_{\mc{U}}(\mc{X}, \omega^{\otimes k}) $$
and we obtain the \emph{descent spectral sequence}\index{descent spectral sequence}
\begin{equation}\label{eq:dss}
E_2^{s,2k} = H^s(\mc{X}, \omega^{\otimes k}) \Rightarrow \pi_{2k-s} \mc{O}^{top}(\mc{X}).
\end{equation}

\subsection*{Non-connective topological modular forms}

Motivated by (\ref{eq:MF}) and (\ref{eq:dss}), we make the following definition.   

\begin{definition}
The spectrum of \emph{(non-connective) topological modular forms}\index{topological modular forms} \index{$\Tmf$} is defined to be the spectrum of global sections
$$ \Tmf := \mc{O}^{op}(\br{\mc{M}}_{ell}). $$
\end{definition}

To get a feel for $\Tmf$, we investigate the descent spectral sequence for $\Tmf[1/6]$.  

\begin{proposition}\label{prop:piTmf}
We have\footnote{Here, we use the notation $\frac{\ZZ[1/6][c_4,c_6]}{(c_4^\infty, c_6^\infty)}\{\theta\}$ to mean that $\theta$ is highest degree non-zero class in this divisible pattern.}
$$
H^*((\br{\mc{M}}_{ell})_{\ZZ[1/6]}, \omega^{\otimes *}) = \ZZ[1/6][c_4, c_6]\oplus \frac{\ZZ[1/6][c_4,c_6]}{(c_4^\infty, c_6^\infty)}\{\theta\} 
$$
where
\begin{align*}
c_k & \in H^0((\br{\mc{M}}_{ell})_{\ZZ[1/6]}, \omega^{\otimes k}), \\
\theta & \in H^1((\br{\mc{M}}_{ell})_{\ZZ[1/6]}, \omega^{\otimes -10}).
\end{align*}
Thus there are no possible differentials in the descent spectral sequence, and we have
$$ \pi_*\Tmf[1/6] \cong \ZZ[1/6][c_4, c_6]\oplus \frac{\ZZ[1/6][c_4,c_6]}{(c_4^\infty, c_6^\infty)}\{\theta\} $$
with
\begin{align*}
\abs{c_k} & = 2k, \\
\abs{\theta} & = -21.
\end{align*}
\end{proposition}

\begin{proof}
Every trivializable elliptic curve $C$ over a $\ZZ[1/6]$-algebra $R$ can be embedded in $\PP^2$, where it takes the Weierstrass form (see, for example, \cite[III.1]{SilvermanI})
\begin{equation}\label{eq:Weier}
 C_{c_4, c_6}: y^2 = x^3 -27c_4x-54c_6, \quad c_4, c_6 \in R
\end{equation}
where the \emph{discriminant}\index{discriminant}
$$ \Delta := \frac{c_4^3-c_6^2}{1728} $$
is invertible. The isomorphisms of elliptic curves of this form are all of the form
\begin{align*}
 f_\lambda: C_{c_4, c_6} & \rightarrow C_{c_4', c_6'} \\
 (x,y) & \mapsto (\lambda^2 x, \lambda^3 y)
\end{align*}
with 
\begin{equation}\label{eq:Gmaction}
c_k' = \lambda^k c_k.
\end{equation}
We deduce that
$$ (\mc{M}_{ell})_{\ZZ[1/6]} = \mr{spec}(\ZZ[1/6][c_4, c_6, \Delta^{-1}])\mmod \GG_m $$
where the $\GG_m$-action is given by (\ref{eq:Gmaction}).  The $\GG_m$-action encodes a grading on $\ZZ[1/6][c_4, c_6, \Delta^{-1}]$ where
$$ \deg c_k := k. $$
Using the coordinate $z = y/x$ at $\infty$ for the Weierstrass curve $C_{c_4, c_6}$ (the identity for the group structure), we compute
$$ f^*_{\lambda} dz = \lambda dz. $$
It follows that the cohomology of $\omega^{\otimes k}$ is the $k$th graded summand of the cohomology of the structure sheaf of $\mr{spec}(\ZZ[1/6][c_4, c_6, \Delta^{-1}])$:
\begin{align} 
H^s((\mc{M}_{ell})_{\ZZ[1/6]}, \omega^{\otimes k}) & \cong  
H^{s,k}(\mr{spec}(\ZZ[1/6][c_4, c_6, \Delta^{-1}])) \\
& = \begin{cases}
\ZZ[1/6][c_4, c_6, \Delta^{-1}]_k, & s = 0, \\
0, & s > 0.
\end{cases} 
\label{eq:HsMell}
\end{align}
We extend the above analysis to the compactification $\br{\mc{M}}_{ell}$ by allowing for nodal singularities.\footnote{A curve of the form $C_{c_4, c_6}$ can only have nodal or cuspidal singularities, and the nodal case is a N\'eron $1$-gon.}  A curve $C_{c_4, c_6}$ has a nodal singularity if and only if $\Delta = 0$ and $c_4$ is invertible.  We therefore compute
\begin{multline*}
H^s((\br{\mc{M}}_{ell})_{\ZZ[1/6]}, \omega^{\otimes k}) \\
\cong  
H^{s,k}(\mr{spec}(\ZZ[1/6][c_4, c_6, \Delta^{-1}]) \cup \mr{spec}(\ZZ[1/6][c_4^{\pm}, c_6]))
\end{multline*}
as the kernel and cokernel of the map
$$ 
\begin{array}{c}
\ZZ[1/6][c_4, c_6, \Delta^{-1}] \\
\oplus \\
\ZZ[1/6][c_4^{\pm}, c_6]
\end{array}
\rightarrow \ZZ[1/6][c_4^{\pm}, c_6, \Delta^{-1}].
$$
\end{proof}

The unlocalized cohomology
$$ H^s(\br{\mc{M}}_{ell}, \omega^{\otimes k}) $$
is non-trivial for arbitrarily large values of $s$, but for $s > 1$ consists entirely of $2$- and $3$-torsion, resulting in $2$- and $3$-torsion persisting to $\pi_*\Tmf$.  This will be discussed in more detail in Section~\ref{sec:pitmf}.  

\subsection*{Variants of Tmf}

We highlight two variants of the spectrum $\Tmf$: the \emph{connective} and the \emph{periodic} versions.  One feature of $\pi_*\Tmf[1/6]$ which is apparent in Proposition~\ref{prop:piTmf} is that
$$ \pi_k \Tmf[1/6] = 0,  \quad -20 \le k \le -1. $$
It turns out that this gap in homotopy groups occurs in the unlocalized $\Tmf$ spectrum (see Section~\ref{sec:pitmf}), and the negative homotopy groups of $\Tmf$ are related to the positive homotopy groups of $\Tmf$ by Anderson duality (at least with 2 inverted, see \cite{Stojanoska}).  

We therefore isolate the positive homotopy groups by defining the \emph{connective $\tmf$-spectrum} to be the connective cover
$$ \tmf := \tau_{\ge 0} \Tmf. $$\index{topological modular forms} \index{$\tmf$}

The modular form $\Delta \in MF_{12}$ is not a permanent cycle in the descent spectral sequence for unlocalized $\Tmf$, but $\Delta^{24}$ is.
It turns out that the map
$$ \pi_* \tmf \rightarrow \pi_*\tmf[\Delta^{-24}] $$
is injective.  Motivated by this, we define the \emph{periodic $\TMF$-spectrum} by
$$ \TMF := \tmf[\Delta^{-24}]. $$\index{topological modular forms}\index{$\TMF$}
This spectrum is $\abs{\Delta^{24}} = 576$-periodic.  We have
$$\TMF \simeq \Tmf[\Delta^{-24}] \simeq \mc{O}^{top}(\mc{M}_{ell}) $$
where the last equivalence comes from the fact that $\mc{M}_{ell}$ is the complement of the zero-locus of $\Delta$ in $\br{\mc{M}}_{ell}$.  

Another variant comes from the consideration of level structures.  Given a congruence subgroup $\Gamma \le SL_2(\ZZ)$, one can consider the modular forms of level $\Gamma$ to be those holomorphic functions on the upper half plane which satisfy (\ref{eq:MF}) for all $A \in \Gamma$, and which satisfy a holomorphicity condition at all of the cusps of the quotient
$$ \mc{M}_{ell}(\Gamma)(\CC) = \mc{H}\mmod \Gamma. $$

Integral versions of $\mc{M}_{ell}(\Gamma)$ can be defined by considering moduli spaces of elliptic curves with certain types of \emph{level structures}.\index{level structure}  The most common $\Gamma$ which are considered are:
\begin{align*}
\Gamma_0(N) & := \left\{ A \in SL_2(\ZZ) \: : \: A \equiv 
\begin{pmatrix} * & * \\ 0 & * \end{pmatrix} \pmod N \right\}, \\
\Gamma_1(N)& := \left\{ A \in SL_2(\ZZ) \: : \: A \equiv 
\begin{pmatrix} 1 & * \\ 0 & 1 \end{pmatrix} \pmod N \right\}, \\
\Gamma(N)& := \left\{ A \in SL_2(\ZZ) \: : \: A \equiv 
\begin{pmatrix} 1 & 0 \\ 0 & 1 \end{pmatrix} \pmod N \right\}.
\end{align*}
The corresponding moduli stacks $\mc{M}_{ell}(\Gamma_0(N))$ and $\mc{M}_{ell}(\Gamma_1(N))$ (respectively $\mc{M}_{ell}(\Gamma(N))$) are defined over $\ZZ[1/N]$ (respectively $\ZZ[1/N, \zeta_N]$), with $R$-points consisting of the groupoid of pairs
$$ (C, \eta) $$
where $C$ is an elliptic curve over $R$, and 
$$
\eta = 
\begin{cases}
\text{a cyclic subgroup of $C$ of order $N$}, & \Gamma = \Gamma_0(N), \\
\text{a point of $C$ of exact order $N$}, & \Gamma = \Gamma_1(N), \\
\text{an isomorphism $C[N] \cong \ZZ/N \times \ZZ/N$}, & \Gamma = \Gamma(N).
\end{cases}
$$

In each of these cases, forgetting the level structure results in an \'etale
map of stacks
\begin{equation}\label{eq:forget}
 \mc{M}_{ell}(\Gamma) \rightarrow \mc{M}_{ell}
 \end{equation}
and we define the associated (periodic) spectra of \emph{topological modular forms, with level structure}\index{topological modular forms with level structure} by
\begin{align*}
\TMF_0(N) & := \mc{O}^{top}(\mc{M}_{ell}(\Gamma_0(N))), \\
\TMF_1(N) & := \mc{O}^{top}(\mc{M}_{ell}(\Gamma_1(N))), \\
\TMF(N) & := \mc{O}^{top}(\mc{M}_{ell}(\Gamma(N))). \\
\end{align*}

Compactifications $\br{\mc{M}}_{ell}(\Gamma)$ of the moduli stacks $\mc{M}_{ell}(\Gamma)$ above were constructed by Deligne and Rapoport \cite{DeligneRapoport}.
The extensions of the maps (\ref{eq:forget}) to these compactifications
$$ 
 \br{\mc{M}}_{ell}(\Gamma) \rightarrow \br{\mc{M}}_{ell}
$$
are not \'etale, but they are log-\'etale.
Hill and Lawson have shown that the sheaf $\mc{O}^{top}$ extends to the log-\'etale site of $\br{\mc{M}}_{ell}$ \cite{HillLawson}, allowing us to define corresponding $\Tmf$-spectra by
\begin{align*}
\Tmf_0(N) & := \mc{O}^{top}(\br{\mc{M}}_{ell}(\Gamma_0(N))), \\
\Tmf_1(N) & := \mc{O}^{top}(\br{\mc{M}}_{ell}(\Gamma_1(N))), \\
\Tmf(N) & := \mc{O}^{top}(\br{\mc{M}}_{ell}(\Gamma(N))). \\
\end{align*}

\section{Homotopy groups of $\TMF$ at the primes $2$ and $3$}\label{sec:pitmf}

We now give an overview of the $2$- and $3$-primary homotopy groups of $\TMF$.  Detailed versions of these computations can be found in \cite{Konter}, \cite{Bauer}, and some very nice charts depicting the answers were created by Henriques \cite{Henriques}.  The basic idea is to invoke the descent spectral sequence.  The $E_2$-term is computed by imitating the argument of Proposition~\ref{prop:piTmf}.  The descent spectral sequence does not degenerate $2$ or $3$-locally, and differentials must be deduced using a variety of ad hoc methods similar to those used to compute differentials in the Adams-Novikov spectral sequence.

\subsection*{$3$-primary homotopy groups of $\TMF$}
Every elliptic curve over a $\ZZ_{(3)}$-algebra $R$ can (upon taking a faithfully flat extension of $R$) be put in the form \cite{Bauer}
$$ C'_{a_2, a_4}: y^2= 4x(x^2+a_2x+a_4), \quad a_i \in R $$
with
$$ \Delta = a_4^2(16a_2^2-64a_4) $$
invertible.
The isomorphisms of any such are of the form
\begin{equation}\label{eq:f_r}
\begin{split} 
f_{\lambda,r}: C'_{a_2, a_4} & \rightarrow C'_{a_2', a_4'} \\
 (x,y) & \mapsto (\lambda^2 (x - r), \lambda^3 y)
\end{split}
\end{equation}
with
$$ r^3 + a_2r^2 + a_4 r = 0 $$
and
\begin{equation}\label{eq:etaR3}
\begin{split}
a_2' & = \lambda^2(a_2 + 3r), \\
a_4' & = \lambda^4(a_4 + 2a_2 r + 3r^2).
\end{split}
\end{equation}
Following the template of the proof of Proposition~\ref{prop:piTmf}, we observe that for the coordinate $z = y/x$ at $\infty$, we have
\begin{equation}\label{eq:Df_r}
 f_{\lambda,r}^* dz = \lambda dz.
 \end{equation}
We may therefore use the $\lambda$ factor to compute the sections of $\omega^{\otimes *}$.

Specifically,
by setting $\lambda = 1$, we associate to this data a graded Hopf algebroid $(A',\Gamma')$ with
\begin{align*}
A' & := \ZZ_{(3)}[a_2, a_4, \Delta^{-1}], \quad \abs{a_i} = i, \\
\Gamma' & :=  A'[r]/(r^3 + a_2r^2 + a_4 r), \quad \abs{r} = 2
\end{align*}
with right unit given by (\ref{eq:etaR3}) (with $\lambda = 1$) and coproduct given by the composition of two isomorphisms of the form $f_{1,r}$.

Now consider the cover
$$ U = \mr{Proj}(A') \rightarrow (\mc{M}_{ell})_{\ZZ_{(3)}}. $$
We deduce from (\ref{eq:Df_r}) that 
$$ \omega^{\otimes k}(U) = A'_k $$
and more generally
$$ \omega^{\otimes k}(U^{\times_{\mc{M}_{ell}} (s+1)}) = \left( (\Gamma')^{\otimes_{A'} s} \right)_k. $$
It follows that the \v Cech complex for $\bigoplus_k \omega^{\otimes k}$ associated to the cover $U$
is the cobar complex for the graded Hopf algebroid $(A', \Gamma')$.
We deduce that the $E_2$-term of the descent spectral sequence is given by the cohomology of the Hopf algebroid $(A', \Gamma')$:
$$ H^s((\mc{M}_{ell})_{\ZZ_{(3)}}, \omega^{\otimes k}) = H^{s,k}(A', \Gamma'). $$
One computes (see \cite{Bauer}):

\begin{proposition}
$$ H^{s,k}(A', \Gamma') = \frac{\ZZ_{(3)}[c_4, c_6, \Delta^{\pm}][\beta] \otimes E[\alpha]}{3\alpha, 3\beta, \alpha c_4, \alpha c_6, \beta c_4, \beta c_6, c_6^2 = c_4^3-1728\Delta} $$
where $\beta$ is given by the Massey product
$$ \beta = \bra{\alpha, \alpha, \alpha} $$
and the generators are in bidegrees $(s,k)$: 
\begin{align*}
\abs{c_i} & =  (0,i), & \abs{\Delta} & = (0,12) \\
\abs{\alpha} & = (1,2), & \abs{\beta} & = (2,6).
\end{align*}
\end{proposition}

\begin{figure}
\includegraphics[angle = 90, origin=c, height =.7\textheight]{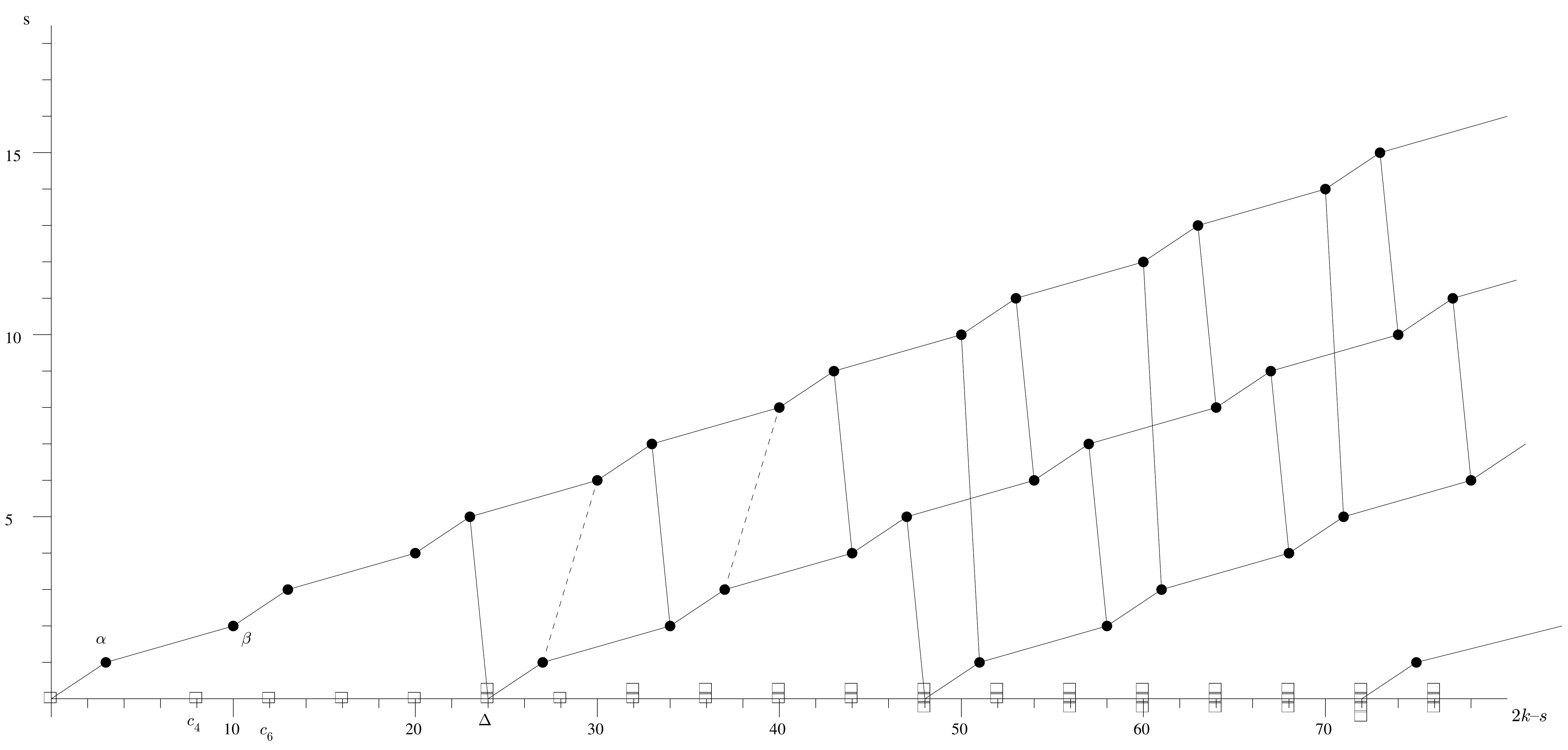}
\caption{The descent spectral sequence for $\tmf_{(3)}$ (the descent spectral sequence for $\TMF_{(3)}$ is obtained by inverting $\Delta$)}\label{fig:tmf3}
\end{figure}

Figure~\ref{fig:tmf3} displays the descent spectral sequence 
$$ H^{s,k}(A', \Gamma') \Rightarrow \pi_{2k-s}\TMF_{(3)}. $$
Here:
\begin{itemize}
\item Boxes correspond to $\ZZ_{(3)}$'s.
\item Dots correspond to $\ZZ/3$'s.
\item Lines of slope $1/3$ correspond to multiplication by $\alpha$.
\item Lines of slope $1/7$ correspond to the Massey product $\bra{-, \alpha, \alpha}$.  
\item Lines of slope $-r$ correspond to $d_r$-differentials.
\item Dashed lines correspond to hidden $\alpha$ extensions.
\end{itemize}
We omit the factors coming from negative powers of $\Delta$.  In other words, the descent spectral sequence for $\TMF$ is obtained from Figure~\ref{fig:tmf3} by inverting $\Delta$.  The differential on $\beta\Delta$ comes from the Toda differential in the Adams-Novikov spectral sequence for the sphere, and this implies all of the other differentials.  As the figure indicates, $\Delta^3$ is a permanent cycle, and so $\pi_*\TMF_{(3)}$ is $72$-periodic.

Under the Hurewicz homomorphism
$$ \pi_*S_{(3)} \rightarrow \pi_*{\TMF_{(3)}} $$
the elements $\alpha_1$ and $\beta_1$ map to $\alpha$ and $\beta$, respectively.

\subsection*{$2$-primary homotopy groups of $\TMF$}

The analysis of the $2$-primary descent spectral sequence proceeds in a similar fashion, except that the computations are significantly more involved.  We will content ourselves to summarize the set-up, and then state the resulting homotopy groups of $\TMF$, referring the reader to \cite{Bauer} for the details.

Every elliptic curve over a $\ZZ_{(2)}$-algebra $R$ can (upon taking an \'etale extension of $R$) be put in the form \cite{Bauer}
$$ C''_{a_1, a_3}: y^2 + a_1 xy + a_3 y = x^3, \quad a_i \in R $$
with
$$ \Delta =  a_3^3(a_1^3 - 27a_3) $$
invertible.

The isomorphisms of any such are of the form
\begin{equation}\label{eq:f_st}
\begin{split} 
f_{\lambda,s,t}: C''_{a_1, a_3} & \rightarrow C''_{a_1', a_3'} \\
 (x,y) & \mapsto (\lambda^2 (x - 1/3(s^2+a_1s)), \lambda^3 (y- sx + 1/3(s^3+a_1 s^2)-t))
\end{split}
\end{equation}
with\footnote{We warn the reader that there may be a typo in the analog of (\ref{eq:rel2}) which appears in \cite[Sec.~7]{Bauer}, as even using (\ref{eq:rel1}), relation (\ref{eq:rel2}) seems to be inconsistent with what appears there.}
\begin{gather}
s^4 - 6st +a_1s^3 - 3a_1t - 3a_3s = 0, \label{eq:rel1}\\
- 27t^2 + 18s^3t + 18a_1s^2t - 27a_3t - 2s^6 - 3a_1s^5 + 9a_3s^3 + a_1^3s^3 + 9a_1a_3s^2 = 0,
\label{eq:rel2}
\end{gather}
and
\begin{equation}\label{eq:etaR2}
\begin{split}
a_1' & := \lambda(a_1+2s), \\
a_3' & := \lambda^3(a_3+ 1/3(a_1 s^2+a_1 s) + 2t).
\end{split}
\end{equation}
Again,
setting $\lambda = 1$, we associate to this data a graded Hopf algebroid $(A'',\Gamma'')$ with
\begin{align*}
A'' & := \ZZ_{(2)}[a_1, a_3, \Delta^{-1}], \quad \abs{a_i} = i, \\
\Gamma'' & :=  A''[s,t]/\sim, \quad \abs{s} = 1, \: \abs{t} = 3
\end{align*}
(where $\sim$ consists of relations (\ref{eq:rel1}), (\ref{eq:rel2})) with right unit given by (\ref{eq:etaR3}) (with $\lambda = 1$) and coproduct given by the composition of two isomorphisms of the form $f_{1,s,t}$.  The $E_2$-term of the descent spectral sequence takes the form
$$ H^s((\mc{M}_{ell})_{(2)}, \omega^{\otimes k}) = H^{s,k}(A'', \Gamma''). $$

\begin{proposition}[\cite{Bauer}, \cite{Rezk}]
The cohomology of the Hopf algebroid $(A'',\Gamma'')$
is given by
$$ H^{*,*}(A'', \Gamma'') = \ZZ_{(2)}[c_4, c_6, \Delta^{\pm}, \eta, a_1^2 \eta, \nu, \epsilon, \kappa, \bar{\kappa}]/(\sim) $$
where $\sim$ consists of the relations
\begin{gather*}
2\eta, \: \eta\nu, \: 4\nu, \: 2\nu^2, \: \nu^3 = \eta\epsilon, \\
2\epsilon, \: \nu \epsilon, \: \epsilon^2, \: 2a_1^2\eta, \: \nu a_1^2\eta, \: \epsilon a_1^2\eta, \: (a_1^2 \eta)^2 = c_4 \eta^2 , \\
2\kappa, \: \eta^2\kappa, \: \nu^2\kappa = 4\bar{\kappa}, \: \epsilon\kappa, \: \kappa^2, \: \kappa a_1^2\eta, \\
\nu c_4, \: \nu c_6, \: \epsilon c_4, \: \epsilon c_6, \: a_1^2\eta c_4 = \eta c_6, \: a_1^2\eta c_6 = \eta c_4^2, \\
\kappa c_4, \: \kappa c_6, \: \bar{\kappa}c_4 = \eta^4\Delta, \: \bar{\kappa}c_6 = \eta^3(a_1^2\eta)\Delta, \:  c_6^2 = c_4^3 - 1728\Delta
\end{gather*}
and the generators are in bidegrees $(s,k)$: 
\begin{align*}
\abs{c_i} & =  (0,i), & \abs{\Delta} & = (0,12), & \abs{\eta} & = (1,1), \\
\abs{a_1^2\eta} & = (1,3), & \abs{\nu} & = (1,2), & \abs{\epsilon} & = (2,5), \\
\abs{\kappa} & = (2,8), & \abs{\bar{\kappa}} & = (4,12).
\end{align*}
\end{proposition}

\begin{figure}
\includegraphics[angle = 90, origin=c, height =.7\textheight]{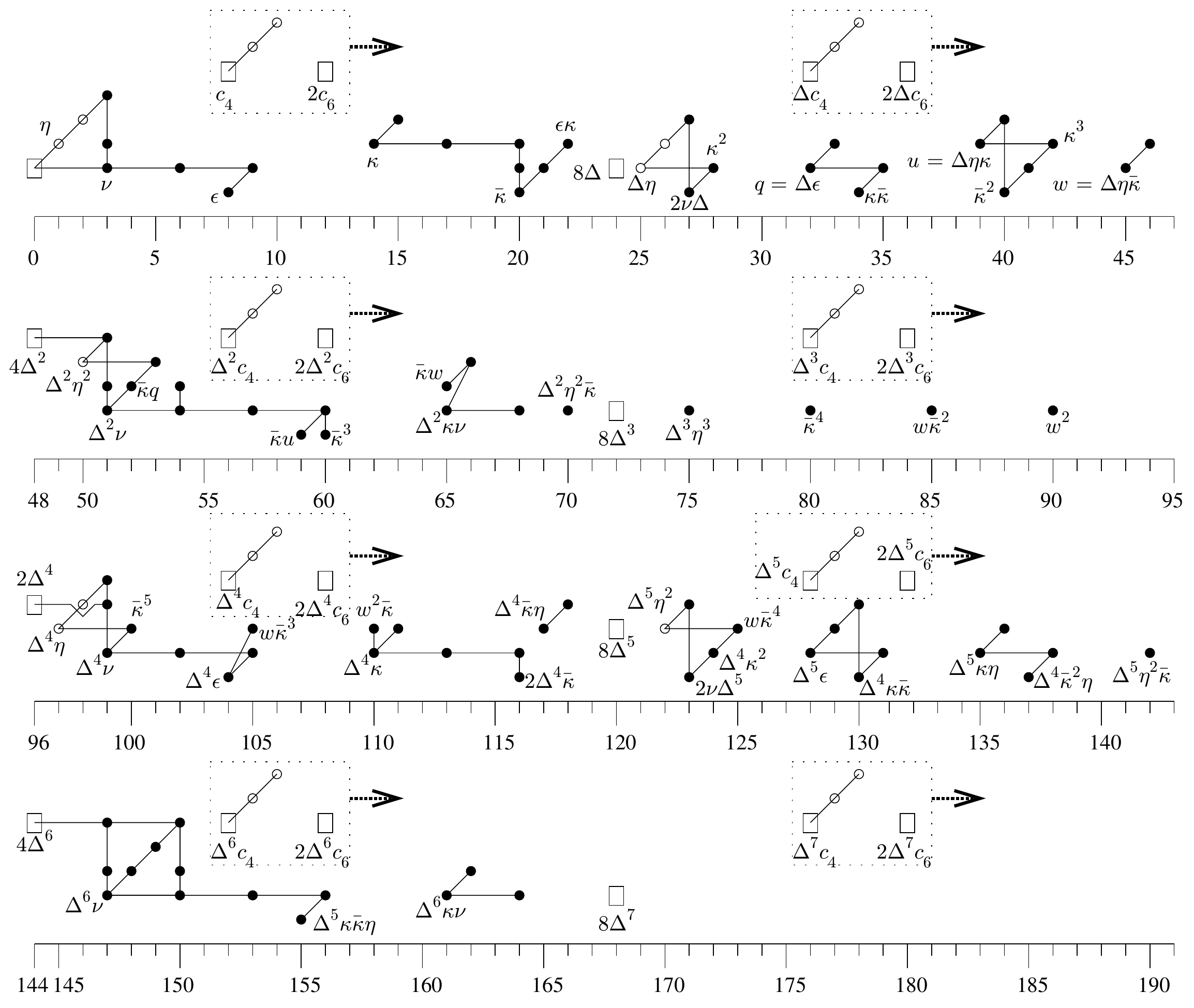}
\caption{The homotopy groups of $\tmf_{(2)}$ ($\pi_*\TMF_{(2)}$ is obtained by inverting $\Delta$)}\label{fig:tmf2}
\end{figure}

There are many differentials in the descent spectral sequence
$$ H^{s,k}(A'', \Gamma'') \Rightarrow \pi_{2k-s}\TMF_{(2)}. $$
These were first determined by Hopkins, and first appeared in the preprint ``From elliptic curves to homotopy theory'' by Hopkins and Mahowald \cite{HopkinsMahowald}, and we refer the reader to that paper or  \cite{Bauer} for the details.

We content ourselves with simply stating the resulting homotopy groups of $\TMF_{(2)}$.  These are displayed in Figure~\ref{fig:tmf2}.  Our choice of names for elements in the descent spectral sequence (and our abusive practice of giving the elements of $\pi_*\TMF$ they detect the same names) is motivated by the fact that the elements
$$ \eta, \nu, \epsilon, \kappa, \bar{\kappa}, q, u, w $$
in the 2-primary stable homotopy groups of spheres map to the corresponding elements in $\pi_*\TMF_{(2)}$.  We warn the reader that there are many hidden extensions in the descent spectral sequence, so that often the names of elements in Figure~\ref{fig:tmf2} do not reflect the element which detects them in the descent spectral sequence because in the descent spectral sequence the product would be zero.  For example, $\kappa^2$ is zero in $H^{*,*}(A'', \Gamma'')$, but nonzero in $\pi_*\TMF$.  More complete multiplicative information can be found in \cite{Henriques}.

In Figure~\ref{fig:tmf2}:
\begin{itemize}
\item A series of $i$ black dots joined by vertical lines corresponds to a factor of $\ZZ/2^i$ which is annihilated by some power of $c_4$.

\item An open circle corresponds to a factor of $\ZZ/2$ which is not annihilated by a power of $c_4$.

\item A box indicates a factor of $\ZZ_{(2)}$ which is not annihilated by a power of $c_4$.

\item The non-vertical lines indicate multiplication by $\eta$ and $\nu$.

\item A pattern with a dotted box around it and an arrow emanating from the right face indicates this pattern continues indefinitely to the right by $c_4$-multiplication (i.e. tensor the pattern with $\ZZ_{(2)}[c_4]$).
\end{itemize}
The element $\Delta^{8}$ is a permanent cycle, and $\pi_*\TMF_{(2)}$ is $192$-periodic on the pattern depicted in Figure~\ref{fig:tmf2}.  The figure does not depict powers of $c_4$ supported by negative powers of $\Delta$.  

\section{The homotopy groups of $\Tmf$ and $\tmf$}\label{sec:piTmf}

We give a brief discussion of how the analysis in Section~\ref{sec:pitmf}  can be augmented to determine $\pi_*\Tmf$, and thus $\pi_*\tmf$.  We refer the reader to \cite{Konter} for more details.  We have already described $\pi_*\Tmf[1/6]$ in Section~\ref{sec:tmf}, so we focus on $\pi_*\Tmf_{(p)}$ for $p = 2,3$.

\subsection*{The ordinary locus}

We first must describe a cover of $\br{\mc{M}}_{ell}$.  We recall that for a formal group $F$, the \emph{$p$-series}\index{p-series} is the formal power series
$$ [p]_F(x) = \underbrace{x+_F \cdots +_F x}_{p} $$
where $x+_Fy := F(x,y)$.  If $F$ is defined over a ring of characteristic $p$, we say it has \emph{height $n$}\index{height, of formal group} if its $p$-series takes the form
$$ [p]_F(x) = v_n^F x^{p^n} + \cdots $$
with $v^F_n$ a unit. 

Elliptic curves (over fields of characteristic $p$) have formal groups of height $1$ or $2$. 
We shall call a trivializable elliptic curve over a $\ZZ_{(p)}$-algebra $R$ \emph{ordinary}\index{ordinary elliptic curve} if the formal group $F_{\bar{C}}$ has height $1$ (where $\bar{C}$ the base change of the curve to $R/p$).  Let $(\mc{M}^{ord}_{ell})_{\ZZ_{(p)}}$ denote the moduli stack of ordinary elliptic curves, and define $(\br{\mc{M}}^{ord}_{ell})_{\ZZ_{(p)}}$ to be the closure of 
$(\mc{M}^{ord}_{ell})_{\ZZ_{(p)}}$ in $(\br{\mc{M}}_{ell})_{\ZZ_{(p)}}$.

We have the following lemma (\cite[Sec.~21]{Rezk}).

\begin{lemma}\label{lem:v1}
Let 
\begin{align*}
F' & = F_{\br{C}'_{a_2,a_4}}, \\
F'' & = F_{\br{C}''_{a_1,a_3}},
\end{align*}
denote the formal group laws of the reductions of the elliptic curves $C'_{a_2, a_4}$ and $C''_{a_1,a_3}$ modulo $3$ and $2$, respectively.  Then we have
\begin{align*}
v_1^{F'} & = -a_2, \\
v_1^{F''} & = a_1.
\end{align*}
\end{lemma}

Using the fact that \cite[Prop.18.7]{Rezk}
\begin{equation}\label{eq:c4}
c_4  \equiv 
\begin{cases}
16a_2^2 \pmod 3, & p = 3, \\
a_1^4 \pmod 2, & p = 2
\end{cases}
\end{equation}
we have (for $p = 2$ or $3$)
$$ \TMF^{ord}_{(p)} = \TMF_{(p)}[c_4^{-1}]. $$
We deduce from the computations of $\pi_*\TMF_{(p)}$:

\begin{proposition}
We have
$$
\pi_*\TMF^{ord}_{(p)} = \begin{cases}
\frac{\ZZ_{(3)}[c_4^{\pm}, c_6, \Delta^{\pm}]}{c_6^2 = c_4^3-1728\Delta}, & p = 3, \\
\frac{\ZZ_{(2)}[c_4^{\pm}, 2c_6, \Delta^{\pm}, \eta]}{2\eta, \eta^3,  \eta\cdot (2c_6), (2c_6)^2 = 4(c_4^3-1728\Delta)}, & p = 2.
\end{cases} 
$$
\end{proposition}

Using the covers
\begin{align*}
\mr{Proj}(\ZZ_{(3)}[a_2^{\pm}, a_4]) \rightarrow (\br{\mc{M}}^{ord}_{ell})_{(3)}, \\
\mr{Proj}(\ZZ_{(2)}[a_1^{\pm}, a_3]) \rightarrow (\br{\mc{M}}^{ord}_{ell})_{(2)}, \\
\end{align*}
the Hopf algebroids $(A',\Gamma')$ and $(A'', \Gamma'')$ have variants where $a_2$ (respectively $a_1$) is inverted and $\Delta$ is not.  Using these, one computes the descent spectral sequence for $\Tmf^{ord}_{(p)}$ at $p = 2, 3$ and finds:

\begin{proposition}
We have:
$$
\pi_*\Tmf^{ord}_{(p)} = \begin{cases}
\frac{\ZZ_{(3)}[c_4^{\pm}, c_6, \Delta]}{c_6^2 = c_4^3-1728\Delta}, & p = 3, \\
\frac{\ZZ_{(2)}[c_4^{\pm}, 2c_6, \Delta, \eta]}{2\eta, \eta^3,  \eta\cdot (2c_6), (2c_6)^2 = 4(c_4^3-1728\Delta)}, & p = 2.
\end{cases}
$$ 
\end{proposition}

\subsection*{A homotopy pullback for $\Tmf_{(p)}$}

The spectrum $\Tmf$ can be accessed at the primes $2$ and $3$ in a manner analogous to the case of Proposition~\ref{prop:piTmf}: associated to the cover
$$ \left\{ (\br{\mc{M}}_{ell}^{ord})_{\ZZ_{(p)}}, (\mc{M}_{ell})_{\ZZ_{(p)}} \right\} \rightarrow (\br{\mc{M}}_{ell})_{\ZZ_{(p)}} $$
there is a homotopy pullback (coming from the sheaf condition of $\mc{O}^{top}$)
\begin{equation}\label{eq:Tmfpullback}
\xymatrix{
\Tmf_{(p)} \ar[r] \ar[d] & 
\mc{O}^{top}((\mc{M}_{ell})_{\ZZ_{(p)}}) \ar[d] \\
\mc{O}^{top}((\br{\mc{M}}_{ell}^{ord})_{\ZZ_{(p)}}) \ar[r] &
\mc{O}^{top}((\mc{M}^{ord}_{ell})_{\ZZ_{(p)}}).
}
\end{equation}
Since we have described the homotopy groups of the spectra
\begin{align*}
\TMF_{(p)} & := \mc{O}^{top}((\mc{M}_{ell})_{\ZZ_{(p)}}), \\
\TMF^{ord}_{(p)} & := \mc{O}^{top}(({\mc{M}}_{ell}^{ord})_{\ZZ_{(p)}}), \\
\Tmf^{ord}_{(p)} & := \mc{O}^{top}((\br{\mc{M}}^{ord}_{ell})_{\ZZ_{(p)}})
\end{align*}
at the primes $2$ and $3$, the homotopy groups of $\Tmf_{(p)}$ at these primes may be computed using the pullback square (\ref{eq:Tmfpullback}).  

\subsection*{The homotopy groups of $\tmf_{(p)}$}

Once one computes $\pi_*\Tmf_{(p)}$ it is a simple matter to read off the homotopy groups of the connective cover $\tmf_{(p)}$. We obtain:

\begin{theorem}
The homotopy groups of $\tmf_{(3)}$ are given by the $E_\infty$-page of the spectral sequence of Figure~\ref{fig:tmf3}, and the homotopy groups of $\tmf_{(2)}$ are depicted in Figure~\ref{fig:tmf2}.  These homotopy groups are $\Delta^3$ (respectively $\Delta^8$)-periodic.
\end{theorem}

We end this section by stating a very useful folklore theorem which was proven rigorously in \cite{Mathew}.

\begin{theorem}[Mathew]
The mod $2$ cohomology of $\tmf$ is given (as a module over the Steenrod algebra) by
$$ H^*(\tmf; \FF_2) = A\mmod A(2) $$
where $A$ is the mod $2$ Steenrod algebra, and $A(2)$ is the subalgebra generated by $\Sq^1$, $\Sq^2$, and $\Sq^4$. 
\end{theorem}

\begin{corollary}
For a spectrum $X$, the Adams spectral sequence for the $2$-adic $\tmf$-homology of $X$ takes the form
$$ \Ext^{s,t}_{A(2)}(H^*(X; \FF_2), \FF_2) \Rightarrow \pi_{t-s}(\tmf \wedge X)^{\wedge}_2. $$ 
\end{corollary}

\section{Tmf from the chromatic perspective}\label{sec:chromatic}

We outline the essential algebro-geometric ideas behind chromatic homotopy theory, as originally envisioned by Morava \cite{Morava} (see also \cite{coctalos}, \cite[Ch.~9]{tmf}, \cite{Goerss}), and apply it to understand the chromatic localizations of $\Tmf$.  We will find that the pullback (\ref{eq:Tmfpullback}) used to access $\Tmf$ is closely related to its chromatic fracture square.

\subsection*{Stacks associated to ring spectra}

The perspective in this section is closely aligned with that of Mike Hopkins's lecture ``From spectra to stacks'' in \cite[Ch.~9]{tmf}.
The complex cobordism spectrum $\MU$ has a canonical complex orientation. 
To conform better to the even periodic set-up, we utilize the even periodic variant\footnote{Just as $\MU$ is the Thom spectrum of the universal virtual bundle over $BU$, $\MUP$ is the Thom spectrum of the universal virtual bundle over $BU \times \ZZ$.}
$$ \MUP := \bigvee_{i \in \ZZ} \Sigma^{2i} \MU $$
so that $\pi_0 \MUP \cong \pi_*MU$.
Quillen proved \cite{Quillen} (see also \cite{Adams}) that the associated formal group law $F_{\MUP}$ is the universal formal group law:
$$ \mr{spec}(\pi_0\MUP)(R) = \{ \text{formal group laws over $R$} \} $$  
In particular
$$ \mr{spec} (\pi_0\MUP) \rightarrow \mc{M}_{fg} $$
is a flat cover.  In fact, we have \cite{Quillen}, \cite{Adams}
\begin{align*}
 \mr{spec}(\MUP_0\MUP) 
 & = \{ \text{isomorphisms $f: F \rightarrow F'$ between formal group laws over $R$} \} \\
 & = \mr{spec}(\pi_0\MUP) \times_{\mc{M}_{fg}} \mr{spec}(\pi_0 \MUP). 
 \end{align*}

Suppose that $E$ is a complex oriented even periodic ring spectrum whose formal group law classifying map
$$ \mr{spec}(\pi_0E) \xrightarrow{F_{E}} \mc{M}_{fg}. $$
is flat.  We shall call such ring spectra \emph{Landweber exact}\index{Landweber exact ring spectrum} (see Theorem~\ref{thm:LEFT}).  The formal group law $F_E$ of such $E$ determines $E$ in the following sense: the classifying map
$$ \pi_0 \MUP \xrightarrow{F_{E}} \pi_0E $$
lifts to a map of ring spectra
$$ \MUP \rightarrow E $$
and the associated map
\begin{equation}\label{eq:LEFT}
\pi_0E \otimes_{\pi_0 \MUP} \MUP_{*}X \rightarrow E_*X
\end{equation}
is an isomorphism for all spectra $X$ \cite{Landweber}, \cite{coctalos}.

We shall say that a commutative ring spectrum $E$ is \emph{even} if 
$$ \MUP_{\mr{odd}}E = 0. $$
It follows that $\MUP \wedge E$ is an even periodic ring spectrum, and in particular is complex orientable.  Moreover, the canonical complex orientation on $\MUP$ induces one on $\MUP \wedge E$.  

We shall say that an even ring spectrum $E$ is \emph{Landweber}\index{Landweber ring spectrum} if
the associated classifying map
$$ \mr{spec}(\MUP_0E) \xrightarrow{F_{\MUP \wedge E}} \mc{M}_{fg} $$
is flat (i.e. $\MUP \wedge E$ is Landweber exact).  It follows from (\ref{eq:LEFT}) that $\MUP_0E$ is an $\MUP_0\MUP$-comodule algebra, and hence $$ \mr{spec} \MUP_0E \rightarrow \mr{spec} \MUP_0 $$ 
comes equipped with descent data to determine a stack
\begin{equation}\label{eq:X_E}
 \mc{X}_E \rightarrow \mc{M}_{fg}.
\end{equation}
We shall call $\mc{X}_E$ the \emph{stack associated to $E$}.\index{stack, associated to ring spectrum}  Let $\omega$ denote the line bundle over $\mc{M}_{fg}$ whose fiber over a formal group law $F$ is the cotangent space at the identity
$$ \omega_F = T^*_0F. $$
We abusively also let $\omega$ denote the pullback of this line bundle to $\mc{X}_E$ under (\ref{eq:X_E}).  Then an analysis similar to that of Section~\ref{sec:pitmf} (see \cite{Devalapurkar}) shows that the spectral sequence associated to the canonical Adams-Novikov resolution 
$$ E^{\wedge}_{MUP} := \mr{Tot} \left( \MUP \wedge E \Rightarrow \MUP \wedge \MUP \wedge E \Rrightarrow \cdots  \right) $$
takes the form
$$  H^s(\mc{X}_{E}, \omega^{\otimes k}) \Rightarrow \pi_{2k-s}E^{\wedge}_{MUP}. $$

\begin{example}
The spectrum $\TMF$ is Landweber \cite[Sec.~20]{Rezk}, \cite[Sec.~5.1]{Mathew}, with
$$ \mc{X}_{\TMF} = \mc{M}_{ell}. $$
The Adams-Novikov spectral sequence is the descent spectral sequence.  In fact, the computations of \cite[Sec.~20]{Rezk} also show that $\tmf$ is Landweber, with
$$ \mc{X}_{\tmf} = \mc{M}_{\mr{weier}}. $$
This is the moduli stack of \emph{Weierstrass curves},\index{Weierstrass curve} curves which locally take the form
$$  y^2 + a_1 xy + a_3 y = x^3 + a_2 x^2 + a_4x + a_6. $$
The associated Adams-Novikov spectral sequence takes the form
$$ H^s(\mc{M}_{weier}, \omega^{\otimes k}) \Rightarrow \pi_{2k-s}\tmf $$
is computed in \cite{Bauer}.  The spectra $\Tmf^{ord}_{(p)}$ and $\TMF^{ord}_{(p)}$ are also Landweber, with 
\begin{align*} 
\mc{X}_{\Tmf^{ord}_{(p)}} & = (\br{\mc{M}}^{ord}_{ell})_{\ZZ_{(p)}}. \\
\mc{X}_{\TMF^{ord}_{(p)}} & = (\mc{M}^{ord}_{ell})_{\ZZ_{(p)}}.
\end{align*}
\end{example}

Unfortunately, the spectrum $\Tmf$ is not Landweber, but the pullback (\ref{eq:Tmfpullback}) does exhibit it as a pullback of Landweber ring spectra.  The pushout of the corresponding diagram of stacks
$$
\xymatrix{
\mc{X}_{\TMF^{ord}_{(p)}} \ar[r] \ar[d] & \mc{X}_{\TMF_{(p)}} \ar[d]  \\
\mc{X}_{\Tmf^{ord}_{(p)}} \ar[r] & (\br{\mc{M}}_{ell})_{(\ZZ_{(p)})}
}
$$
motivates us to consider $(\br{\mc{M}}_{ell})_{(\ZZ_{(p)})}$ as the appropriate stack $\mc{X}_{\Tmf_{(p)}}$ over $\mc{M}_{ell}$ to associate to $\Tmf_{(p)}$.  This motivates the following definition.

\begin{definition}
We shall call a ring spectrum $E$ \emph{locally Landweber}\index{locally Landweber ring spectrum} if it is given as a homotopy limit
$$ E = \mr{holim}_{i\in \mc{I}} E_i $$
of Landweber ring spectra where $\mc{I}$ is a category whose nerve has finitely many non-degenerate simplices.\footnote{We specify this condition so that homotopy limits taken over $\mc{I}$ commute with homotopy colimits in the category of spectra --- see the proof of Proposition~\ref{prop:Landwebersmash}.} The colimit
$$ \mc{X}_{E} := \colim_i \mc{X}_{E_i} $$
is the \emph{stack associated to $E$}. 
\end{definition}

\begin{remark}
The stack $\mc{X}_E$ in the above definition a priori seems to to depend on the diagram $\{E_i\}$.  In general, the $E_2$-term of the Adams-Novikov spectral sequence for $E$ is \emph{not} isomorphic to $H^s(\mc{X}_E, \omega^{\otimes k})$ (as happens in the case of Landweber spectra).  \end{remark}

\begin{proposition}\label{prop:Landwebersmash}
Suppose that $E$ and $E'$ are locally Landweber.  Then so is $E \wedge E'$, and 
$$ \mc{X}_{E \wedge E'} \simeq \mc{X}_{E} \times_{\mc{M}_{fg}} \mc{X}_{E'}. $$
\end{proposition}

\begin{proof}
Suppose first that $E$ and $E'$ are Landweber.  The result then follows from the fact that we have
\begin{align*}
\MUP_0(E \wedge E') & \cong (\MUP \wedge E)_0(E') \\
& \cong (\MUP \wedge E)_0 \otimes_{\MUP_0} \MUP_0(E') \\
& = \MUP_0(E) \otimes_{\MUP_0} \MUP_0(E').
\end{align*}
Now suppose that $E$ and $E'$ are locally Landweber, given as limits
\begin{align*}
E & \simeq \holim_{i \in \mc{I}} E_i, \\
E' & \simeq \holim_{j \in \mc{J}} E'_j.
\end{align*}
Then the finiteness conditions on $\mc{I}$ and $\mc{J}$ allow us to compute
\begin{align*}
\holim_{i,j} (E_i \wedge E'_j) & \simeq (\holim_i E_i) \wedge (\holim_j E'_j) \\
& \simeq E \wedge E'_j
\end{align*}
and we have
\begin{align*}
 \mc{X}_{E} \times_{\mc{M}_{fg}} \mc{X}_{E'} = & (\colim_i \mc{X}_{E_i}) \times_{\mc{M}_{fg}} (\colim_j \mc{X}_{E_j}) \\
 & = \colim_{i,j} \mc{X}_{E_i} \times_{\mc{M}_{fg}} \mc{X}_{E_j} \\
 & \simeq \colim_{i,j} \mc{X}_{E_i \wedge E'_j} \\
 & = \mc{X}_{E \wedge E'}.
 \end{align*}
\end{proof}

\subsection*{The stacks associated to chromatic localizations}

Let
$$ (\mc{M}_{fg})^{\le n}_{\ZZ_{(p)}} \subset \mc{M}_{fg} $$
denote the substack which classifies $p$-local formal group laws of height $\le n$.  Let 
$$ (\mc{M}_{fg})^{[n]}_{\ZZ_{(p)}} \subset (\mc{M}_{fg})^{\le n}_{\ZZ_{(p)}} $$ 
denote the formal neighborhood\footnote{$(\mc{M}_{fg})^{[n]}_{\ZZ_{(p)}}$ is technically a \emph{formal stack}.} of the locus of formal group laws in characteristic $p$ of exact height $n$.  

Over $\bar{\FF}_p$, any two formal groups of height $n$ are isomorphic.
Lubin and Tate showed that there is a formally affine Galois cover\footnote{The cover $\mc{X}_n$ depends on a choice of height $n$ formal group over $\FF_{p^n}$, but we suppress the role of this choice to simplify the exposition.}
$$ \mr{Spf}(\ZZ_p[\zeta_{p^n-1}][[u_1, \ldots, u_{n-1}]]) = \mc{X}_n \rightarrow (\mc{M}_{fg})^{[n]}_{\ZZ_{(p)}} $$
with profinite Galois group
$$ \GG_n = \MS_n \rtimes \mr{Gal}(\FF_{p^n}/\FF_p) $$
where $\MS_n$ is the \emph{Morava stabilizer group}.\index{Morava stablizer group}

Fix a prime $p$, and let $E(n)$ denote the $n$th Johnson-Wilson spectrum,\index{Johnson-Wilson spectrum}\index{$E(n)$} $K(n)$ the $n$th Morava $K$-theory spectrum,\index{Morava K-theory}\index{$K(n)$} and $E_n$ the $n$th Morava $E$-theory spectrum,\index{Morava E-theory}\index{$E_n$} with
\begin{align*}
\pi_* E(n) & = \ZZ_{(p)}[v_1, \ldots, v_{n-1}, v_n^{\pm}], \\
\pi_* K(n) & = \FF_p[v_n^{\pm}], \\
\pi_* E_n & = \ZZ_p[\zeta_{p^n-1}][[u_1, \ldots, u_{n-1}][u^{\pm}].
\end{align*}
Here, $\abs{v_i} = 2(p^i-1)$, $\abs{u_i} = 0$, and $\abs{u} = -2$.

For spectra $X$ and $E$, let $X_E$ denote the $E$-localization of the spectrum $X$ \cite{Bousfield}.
Then we have the following correspondence between locally Landweber spectra and associated stacks over $\mc{M}_{fg}$:
\vspace{12pt}

\begin{tabular}{c|c}
spectrum $E$ & stack $\mc{X}_E$ \\
\hline
\\
$S$ & $\mc{M}_{fg}$ \\ \\
$S_{E(n)}$ & $(\mc{M}_{fg})^{\le n}_{\ZZ_{(p)}}$  \\ \\
$S_{K(n)}$ & $(\mc{M}_{fg})^{[n]}_{\ZZ_{(p)}}$   \\ \\
$E_n$ & $\mc{X}_n$ \\ \\
\end{tabular} 
\vspace{12pt}

\begin{remark}
The spectrum $E = S_{K(n)}$ is really only Landweber in the $K(n)$-local category, in the sense that $(\MUP \wedge E)_{K(n)}$ is Landweber exact.  However, similar considerations associate formal stacks $\mc{X}_E$ to such $K(n)$-local ring spectra, and an analog of Proposition~\ref{prop:Landwebersmash} holds where
$$ \mc{X}_{(E \wedge E')_{K(n)}} \simeq \mc{X}_{E} \widehat{\times}_{\mc{M}_{fg}} \mc{X}_{E'}. $$
The spectrum $S_{E(n)}$ is a limit of spectra which are Landweber in the above $K(i)$-local sense $i \le n$ \cite{CamarenaBarthel}.
\end{remark}

The spectrum $E_n$ is in fact Landweber exact.
Galois descent is encoded in the work of Goerss-Hopkins-Miller \cite{GoerssHopkins} and Devinatz-Hopkins \cite{DevinatzHopkins}, who showed that the group $\GG_n$ acts on $E_n$ with
$$ S_{K(n)} \simeq E_n^{h\GG_n}. $$

The following proposition follows from Proposition~\ref{prop:Landwebersmash} (or its $K(n)$-local variant), and is closely related to localization forulas which appear in \cite{GreenleesMay}.

\begin{proposition}\label{prop:localizedLandweber}
Suppose that $E$ is locally Landweber.  Then so is $E_{E(n)}$ and $E_{K(n)}$, and the associated stacks are given as the pullbacks 
$$
\xymatrix{
\mc{X}_{E_{E(n)}} \ar[r] \ar[d] &
(\mc{M}_{fg})^{\le n}_{\ZZ_{(p)}} \ar[d] &
\mc{X}_{E_{K(n)}} \ar[r] \ar[d] &
(\mc{M}_{fg})^{[n]}_{\ZZ_{(p)}} \ar[d] 
\\
\mc{X}_E \ar[r] &
\mc{M}_{fg} &
\mc{X}_E \ar[r] &
\mc{M}_{fg}
}
$$
and 
$$ E_{K(n)} = (E_{E(n)})^{\wedge}_{I_n} $$
where $I_n = (p, v_1, \ldots, v_{n-1})$ is the ideal corresponding to the locus of height $n$ formal groups in $(\mc{M}_{fg})^{\le n}_{\ZZ_{(p)}}$.
\end{proposition}

For a general spectrum $X$, the square
$$
\xymatrix{
X_{E(n)} \ar[d] \ar[r] &  X_{K(n)} \ar[d] \\
X_{E(n-1)} \ar[r] & (X_{K(n)})_{E(n-1)}.
}
$$ 
is a homotopy pullback (the \emph{chromatic fracture square}).\index{chromatic fracture square}  If $X$ is a locally Landweber ring spectrum, the chromatic fracture square can be regarded as the being associated to
the ``cover''
$$ \left\{ (\mc{M}_{fg})^{\le n-1}_{\ZZ_{(p)}}, (\mc{M}_{fg})^{[n]}_{\ZZ_{(p)}} \right\}  \rightarrow (\mc{M}_{fg})^{\le n}_{\ZZ_{(p)}}. $$

\subsection*{$K(1)$-local $\Tmf$}

Applying Proposition~\ref{prop:localizedLandweber}, we find
$$ \Tmf_{K(1)} \simeq \tmf_{K(1)} \simeq (\Tmf^{ord}_{(p)})^{\wedge}_p. $$

We explain the connection of $K(1)$-local $\Tmf$ to the Katz-Serre theory of $p$-adic modular forms.

The ring of \emph{divided congruences}\index{divided congruences, ring of} is defined to be
$$ D := \left\{ \sum f_k \in \bigoplus_k MF_k \otimes \QQ \: : \: \sum f_k(q) \in \ZZ[[q]] \right\}. $$
This ring was studied extensively by Katz \cite{Katzdivided}, who showed in \cite{Katz} that there is an isomorphism
$$ D \cong \Gamma \mc{O}_{\br{\mc{M}}^{triv}_{ell}} $$
where $\br{\mc{M}}^{triv}_{ell}$ is the pullback
$$
\xymatrix{
\br{\mc{M}}^{triv}_{ell} \ar[r] \ar[d] & \br{\mc{M}}_{ell} \ar[d] \\
\mr{spec}(\ZZ) \ar[r]_{F_{\GG_m}} & \mc{M}_{fg}
}
$$

Since complex $K$-theory is the Landweber exact ring spectrum associated to $F_{\GG_m}$,  Proposition~\ref{prop:Landwebersmash} recovers the following theorem of Laures \cite{Laures}.

\begin{proposition}[Laures]
The complex $K$-theory of $\Tmf$ is given by
$$ K_0\Tmf \cong D. $$
\end{proposition}

The ring of \emph{generalized $p$-adic modular functions}\index{p-adic modular forms} \cite{Katz}, \cite{Katzdivided} is the $p$-completion of the ring $D$:
$$ V_p := D^{\wedge}_p $$  
and the proposition above implies that there is an isomorphism
$$ \pi_0 (K \wedge \Tmf)^{\wedge}_p \cong V_p. $$
For $p \nd \ell$, the action of the $\ell$th Adams operation $\psi^{\ell}$ on this space coincides with the action of $\ell \in \ZZ^\times_p$ on $V$ described in \cite{Katzdivided}, and
$$ V_p\bra{k} = \{ f \in V \: : \: \psi^\ell f = \ell^k f \} $$
is isomorphic to Serre's space of $p$-adic modular forms of weight $k$.  

Letting $p$ be odd, and 
choosing $\ell$ to be a topological generator of $\ZZ_p^\times$, we deduce from the fiber sequence \cite[Lem.~2.5]{HopkinsMahowaldSadofsky}
$$ S^{2k}_{K(1)} \rightarrow K^{\wedge}_p \xrightarrow{\psi^\ell - \ell^k} K^{\wedge}_p $$
the following theorem of Baker \cite{Baker}.

\begin{proposition}
For $p \ge 3$, the homotopy groups of $\Tmf_{K(1)}$ are given by the spaces of $p$-adic modular forms
$$ \pi_{2k} \Tmf_{K(1)} \cong V_p\bra{k}. $$
\end{proposition}

\begin{remark}
At the prime $p = 2$, Hopkins \cite{HopkinsK1} and Laures \cite{LauresK1} studied the spectrum $\Tmf_{K(1)}$, and showed that it has a simple construction as a finite cell object in the category of $K(1)$-local $E_\infty$-ring spectra.
\end{remark}

\subsection*{$K(2)$-local $\Tmf$}

An elliptic curve $C$ over a field of characteristic $p$ is called \emph{supersingular}\index{supersingular elliptic curve} if its formal group $F_{C}$ has height $2$.  Over $\br{\FF}_p$ there are only finitely many isomorphism classes of supersingular elliptic curves.  We shall let $(\mc{M}^{ss}_{ell})_{\ZZ_{(p)}}$ denote the formal neighborhood of the supersingular locus in $(\mc{M}_{ell})_{\ZZ_{(p)}}$.

Serre-Tate theory \cite{LST} implies that the following square is a pullback, and that for each supersingular elliptic curve $C$ there exists a lift:
$$
\xymatrix{
& \mc{X}_{2} \ar[d] \ar@{.>}[dl]_{\td{C}}
\\
(\mc{M}^{ss}_{ell})_{\ZZ_{(p)}} \ar[r] \ar[d] &
(\mc{M}_{fg})^{[n]}_{\ZZ_{(p)}} \ar[d]  \\
\br{\mc{M}}_{ell} \ar[r] & \mc{M}_{fg}
}
$$
By Proposition~\ref{prop:localizedLandweber}, we have
$$ \mc{X}_{\Tmf_{K(2)}} \simeq \mc{X}_{\tmf_{K(2)}} \simeq (\mc{M}^{ss}_{ell})_{\ZZ_{(p)}}. $$ 
For $p = 2,3$ there is only one supersingular curve, and $\td{C}$ is a Galois cover, with Galois group equal to 
$$ \aut(C) \rtimes \mr{Gal}(\FF_{p^2}/\FF_p) \le \GG_2 $$
where
$$ \abs{\aut(C)}  = \begin{cases} 24, & p = 2, \\ 12, & p = 3. \end{cases} $$
Since the Morava $E$-theory $E_2$ is Landweber exact with $\mc{X}_{E_2} = \mc{X}_2$, we have the following.

\begin{proposition}
For $p = 2,3$ there are equivalences
$$ \Tmf_{K(2)} \simeq \TMF_{K(2)} \simeq \tmf_{K(2)} \simeq E_2^{h\aut(C) \rtimes Gal(\FF_{p^2}/\FF_p)}. $$
\end{proposition}

\begin{remark}
For a maximal finite subgroup $G \le \GG_n$, the homotopy fixed point spectrum spectrum $E_2^{hG}$ is denoted $EO_n$.  Therefore the proposition above is stating that at the primes $2$ and $3$ there is an equivalence
$$ \Tmf_{K(2)} \simeq EO_2. $$
\end{remark}

Finally, we note that combining Proposition~\ref{prop:localizedLandweber} with Lemma~\ref{lem:v1} and (\ref{eq:c4}) yields the following.

\begin{proposition}\label{prop:K2TMF}
There is an isomorphism
$$ \pi_* (\TMF_{K(2)}) \cong \pi_*(\TMF)^{\wedge}_{p,c_4} $$
for $p = 2,3$.
\end{proposition}

\subsection*{Chromatic fracture of $\Tmf$}

Since every elliptic curve in characteristic $p$ has height $1$ or $2$, we deduce that the square
$$
\xymatrix{
(\br{\mc{M}}_{ell})_{\ZZ_{(p)}} \ar[r] \ar[d] & (\mc{M}_{fg})^{\le 2}_{\ZZ_{(p)}} \ar[d]  \\
\br{\mc{M}}_{ell} \ar[r] & \mc{M}_{fg}
}  
$$
is a pullback.  We deduce from Proposition~\ref{prop:localizedLandweber} that $\Tmf_{(p)}$ and $\TMF_{(p)}$ are $E(2)$-local.\footnote{The spectrum $\tmf_{(p)}$ is not $E(2)$-local, as cuspidal Weierstrass curves in characteristic $p$ have formal groups of infinite height.}

The $p$-completion of the chromatic fracture square for $\Tmf$
$$
\xymatrix{
\Tmf^{\wedge}_p \ar[r] \ar[d] & \Tmf_{K(2)} \ar[d] \\
\Tmf_{K(1)} \ar[r] & (\Tmf_{K(2)})_{K(1)}
}
$$
therefore takes the form
\begin{equation}\label{eq:Tmfchromatic}
\xymatrix{
\Tmf^{\wedge}_p \ar[r] \ar[d] & \TMF^{\wedge}_{(p,c_4)} \ar[d] \\
(\Tmf^{ord})^{\wedge}_p \ar[r] & (\TMF^{\wedge}_{c_4}[c_4^{-1}])^{\wedge}_p
}
\end{equation}
and corresponds to the cover
$$ \{ (\mc{M}_{ell}^{ss})^{\wedge}_{p}, (\br{\mc{M}}^{ord}_{ell})^{\wedge}_{p} \} \rightarrow (\br{\mc{M}}_{ell})^{\wedge}_p.
$$
The $p$-completed chromatic fracture square for $\Tmf$ is therefore a completed version of the homotopy pullback (\ref{eq:Tmfpullback}).

\begin{remark}
For $p \ge 5$ we have
$$ v_1 = E_{p-1} $$
where $E_{p-1}$ is the normalized Eisenstein series of weight $p-1$.  We therefore have analogs of Proposition~\ref{prop:K2TMF} and (\ref{eq:Tmfchromatic}) for $p \ge 5$ where we replace $c_4$ with $E_{p-1}$.
\end{remark}

\section{Topological automorphic forms}

\subsection*{$p$-divisible groups}  

Fix a prime $p$.

\begin{definition}
A \emph{$p$-divisible group of height $n$}\index{p-divisible group} over a ring $R$ is a sequence of commutative group schemes
$$ 0 = G_0 \le G_1 \le G_2 \le \cdots $$
so that each $G_i$ is locally free of rank $p^{in}$ over $R$, and such that for each $i$ we have
$$ G_i = \ker ([p^i]:G_{i+1} \rightarrow G_{i+1}) $$
\end{definition}

\begin{example}
Suppose that $A$ is an abelian variety over $R$ of dimension $n$.  Then the sequence of group schemes $A[p^\infty] := \{A[p^i]\}$ given by the $p^i$-torsion points of $A$
$$ A[p^i] := \ker([p^i]: A \rightarrow A) $$ 
is a $p$-divisible group of height $2n$.
\end{example}

\begin{example}
Suppose that $F$ is a formal group law over a $p$-complete ring $R$ of height $n$.  Then the sequence of group schemes $F[p^\infty] = \{ F[p^i] \}$
where 
$$ F[p^i] := \mr{spec}(R[[x]]/([p^i]_F(x))) $$
is a $p$-divisible group of height $n$.
\end{example}

Given a $p$-divisible group $G = \{G_i\}$ of height $n$ over a $p$-complete ring $R$, the formal neighborhood of the identity
$$ F_G \le \colim_{i} G_i $$
is a formal group of height $\le n$ \cite{Messing}.  We define the \emph{dimension}\index{dimension, of p-divisible group} of $G$ to be the dimension of the formal group $F_G$.
We shall say $G$ is trivializable if the line bundle $T^*_0F_G$ is trivial.

If $A$ is an abelian variety of dimension $n$ over $R$, then we have
$$ F_{A[p^\infty]} = F_{A} $$
and the dimension of $A[p^\infty]$ is $n$.  

\subsection*{Lurie's theorem}

Let $\mc{M}^n_{pd}$ denote the moduli stack of $1$-dimensional $p$-divisible groups of height $n$, and let $(\mc{M}^n_{pd})^{DM}_{et}$ denote the site of formally \'etale maps
\begin{equation}\label{eq:Gclass}
\xymatrix{
 \mc{X} \xrightarrow{G} \mc{M}^n_{pd}
 }
\end{equation}
where $\mc{X}$ is a locally Noetherian separated Deligne-Mumford stack over a complete local ring with perfect residue field of characteristic $p$.

\begin{remark}\label{rmk:etale}
One typically checks that a map (\ref{eq:Gclass}) is formally \'etale by checking that for each closed point $x \in \mc{X}$, the formal neighborhood of $x$ is isomorphic to the universal deformation space of the fiber $G_x$.
\end{remark}

Lurie proved the following seminal theorem \cite{LurieI}.

\begin{theorem}[Lurie]\label{thm:Lurie}
There is a sheaf $\mc{O}^{top}$\index{$\mc{O}^{top}$} of $E_\infty$ ring spectra on $(\mc{M}^n_{pd})^{DM}_{et}$ with the following property: the ring spectrum
$$ E := \mc{O}^{top}\left( 
\mr{spec}(R) \xrightarrow{G} \mc{M}^n_{pd}\right) $$
(associated to an affine formal \'etale open with $G$ trivializable) is even periodic, with
$$ F_E = F_G. $$
\end{theorem}

This theorem generalizes the Goerss-Hopkins-Miller theorem \cite{GoerssHopkins}.  Consider the Lubin-Tate universal deformation space
$$ \mc{X}_n \xrightarrow{\td{F}} \mc{M}^{[n]}_{fg}. $$
The map classifying the $p$-divisible group $\td{F}[p^\infty]$ 
$$ \mc{X}_n \xrightarrow{\td{F}[p^\infty]} \mc{M}^n_{pd} $$ 
is formally \'etale, simply because the data of a formal group is the same thing as the data of its associated $p$-divisible group over a $p$-complete ring, so they have the same deformations (see Remark~\ref{rmk:etale}). The associated ring spectrum is Morava $E$-theory:
$$ \mc{O}^{top}(\mc{X}_n) \simeq E_n. $$ 
The functoriality of $\mc{O}^{top}$ implies that $\GG_n$ acts on $E_n$.

Theorem~\ref{thm:Lurie} also generalizes (most of) Theorem~\ref{thm:GHM}.  Serre-Tate theory states that deformations of abelian varieties are in bijective correspondence with deformations of their $p$-divisible groups.  Again using Remark~\ref{rmk:etale}, this implies that the map
\begin{align*}
(\mc{M}_{ell})_{\ZZ_p} & \rightarrow \mc{M}^n_{fg}, \\
C & \mapsto C[p^\infty]
\end{align*}
is formally \'etale.  We deduce the existence of $\mc{O}^{top}$ on $(\mc{M}_{ell})_{\ZZ_p}$.

\subsection*{Cohomology theories associated to certain PEL Shimura stacks}

The main issue which prevents us from associating cohomology theories to general $n$-dimensional abelian varieties is that their associated $p$-divisible groups are not $1$-dimensional (unless $n = 1$, of course).

\emph{PEL Shimura stacks}\index{Shimura stack} are moduli stacks of abelian varieties with the extra structure of \underline{P}olarization, \underline{E}ndomorphisms, and \underline{L}evel structure.  We will now describe a class of PEL Shimura stacks (associated to a rational form of the unitary group $U(1,n-1)$) whose PEL data allow for the extraction of a $1$-dimensional $p$-divisible group satisfying the hypotheses of Theorem~\ref{thm:Lurie}.

In order to define our Shimura stack $\mc{X}_{V,L}$, we need to fix the following data.
\begin{align*}
F & = 
\text{quadratic imaginary extension of $\QQ$, such that $p$ splits as $u
\bar{u}$.}
\\
\mc{O}_F & = 
\text{ring of integers of $F$.}
\\
V & = 
\text{$F$-vector space of dimension $n$.}
\\
\bra{-,-} & = 
\text{$\QQ$-valued non-degenerate hermitian alternating form on $V$} \\
& \qquad \text{(i.e. $\bra{\alpha x, y} = \bra{x, \bar{\alpha} y}$ for $\alpha \in F$) of signature $(1,n-1)$.}\\
L & = 
\text{$\mc{O}_F$-lattice in $V$,
$\bra{-,-}$ restricts to give integer values on $L$,} \\
& \qquad \text{and makes $L_{(p)}$ self-dual.}
\end{align*}

Assume that $S$ is a locally Noetherian scheme on which $p$ is locally
nilpotent.
The groupoid $\mathcal{X}_{V,L}(S)$ consist of
tuples of data
$(A,i,\lambda)$, and isomorphisms of such, defined as 
follows.
\vspace{12pt}

\begin{tabular}{lp{17pc}}
$A$ & is an abelian scheme over $S$ of dimension $n$.\\ \\
$\lambda: A \rightarrow A^\vee$& is a polarization (principle at $p$),
 with Rosati involution $\dag$ on $\mr{End}(A)_{(p)}$.\\ \\
$i: \mc{O}_{F,(p)} \hookrightarrow \mr{End}(A)_{(p)}$ & is an
  inclusion of rings, satisfying $i(\bar{z}) = i(z)^\dag$.
\end{tabular}
\vspace{12pt}

We impose the following two conditions (one at $p$, one away from $p$) which basically amount to saying that the tuple $(A,i,\lambda)$ is locally modeled on $(L, \bra{-,-})$:
\begin{enumerate}
\item The splitting $p = u\bar{u}$ in $\mc{O}_F$ induces a splitting $\mc{O}_{F,p} = \mc{O}_{F,u} \times \mc{O}_{F, \bar{u}}$, and hence a splitting of $p$-divisible groups
$$ A[p^\infty] \cong A[u^\infty] \oplus A[\bar{u}^\infty]. $$
We require that $A[u^\infty]$ is $1$-dimensional.

\item
Choose a geometric point $s$ in each 
component of $S$. We require that for each of these points
there \emph{exists} an $\mc{O}_F$-linear integral uniformization
$$
\eta : \prod_{\ell \ne p}L^{\wedge}_{\ell} \xrightarrow{\cong} \prod_{\ell \ne p} T_{\ell}(A_s)
$$
(where $T_\ell A = \lim_i A[\ell^i]$ is the $\ell$-adic Tate module) which, when tensored with $\QQ$, sends $\bra{-,-}$ to an 
$(\AF^{p,\infty})^\times$-multiple of
the $\lambda$-Weil pairing.\footnote{Here, $\AF^{p,\infty} := \left( \prod_{\ell \ne p} \ZZ_\ell \right)\otimes \QQ$ are the \emph{adeles away from $p$ and $\infty$}.} 
\end{enumerate}

Given a tuple $(A, i, \lambda) \in \mc{X}_{V,L}(S)$, the conditions on $i$ and $\lambda$ imply that the polarization induces an isomorphism
\begin{equation}\label{eq:polarization}
\lambda_*: A[u^\infty] \xrightarrow{\cong} A[\bar{u}^\infty]^\vee
\end{equation}
(where $(-)^\vee$ denotes the Cartier dual).  This implies that the $p$-divisible group $A[u^\infty]$ has height $n$.  Serre-Tate theory \cite{LST} implies that deformations of an abelian variety are in bijective correspondence with the deformations of its $p$-divisible group.  The isomorphism (\ref{eq:polarization}) therefore implies that deformations of  a tuple $(A,i,\lambda)$ are in bijective correspondence with deformations of $A[u^\infty]$.  By Remark~\ref{rmk:etale}, the map
\begin{align*}
\mc{X}_{V,L} & \rightarrow \mc{M}^n_{pd}
\end{align*}
is therefore formally etale.  Applying Lurie's theorem, we obtain

\begin{theorem}[\cite{taf}]
There exists a sheaf $\mc{O}^{top}$\index{$\mc{O}^{top}$} of $E_\infty$ ring spectra on the site $(\mc{X}_{V,L})_{et}$, such that for each affine \'etale open
$$ \mr{spec}(R) \xrightarrow{(A,i,\lambda)} \mc{X}_{V,L} $$
with $A[u^\infty]$ trivializable, the associated ring spectrum
$$ E := \mc{O}^{top}(\mr{spec}(R) \xrightarrow{(A,i,\lambda)} \mc{X}_{V,L})
$$
is even periodic with
$$ F_E = F_{A[u^\infty]}. $$
\end{theorem}

The spectrum of \emph{topological automorphic forms} (TAF)\index{topological automorphic forms} for the Shimura stack $\mc{X}_{V,L}$ is defined to be the spectrum of global sections
$$ \TAF_{V,L} := \mc{O}^{top}(\mc{X}_{V,L}). $$\index{$\TAF$}
Let $\omega$ be the line bundle over $\mc{X}_{V,L}$ with fibers given by
$$ \omega_{A,i,\lambda} = T_0^*F_{A[u^\infty]}. $$
Then the construction of the descent spectral sequence (\ref{eq:dss}) goes through verbatim to give a descent spectral sequence
$$ E_2^{s,2k} = H^s(\mc{X}_{V,L}, \omega^{\otimes k}) \Rightarrow \pi_{2k-s}\TAF_{V,L}. $$
The motivation behind the terminology ``topological automorphic forms'' is that the space of sections
$$ AF_k(U(V), L)_{\ZZ_p} := H^0(\mc{X}_{V,L}, \omega^{\otimes k}) $$
is the space of scalar valued weakly holomorphic automorphic forms for the unitary group $U(V)$ (associated to the lattice $L$) of weight $k$ over $\ZZ_p$.

\begin{remark}
Similar to the modular case, the space of holomorphic automorphic forms has an additional growth condition which is analogous to the requirement that a modular form be holomorphic at the cusp.  The term ``weakly holomorphic'' means that we drop this requirement.  However, for $n \ge 3$, it turns out that every weakly holomorphic automorphic form is holomorphic \cite[Sec.~5.2]{Shimura}. 
\end{remark}

The spectra $\TAF_{V,L}$ are locally Landweber, with
$$ \mc{X}_{\TAF_{V,L}} \simeq \mc{X}_{V,L}. $$
The height of the formal groups $F_{A[u^\infty]}$ associated to mod $p$ points $(A,i,\lambda)$ of the Shimura stack $\mc{X}_{V,L}$ range from $1$ to $n$.  We deduce from Proposition~\ref{prop:localizedLandweber} that $\TAF_{V,L}$ is $E(n)$-local, and an analysis similar to that in the $\Tmf$ case (see Section~\ref{sec:chromatic}) yields the following.

\begin{proposition}[\cite{taf}]
The $K(n)$-localization of $\TAF_{V,L}$ is given by
$$ (\TAF_{V,L})_{K(n)} \simeq \left( \prod_{(A,i,\lambda)} E^{h\aut(A,i,\lambda)}_n \right)^{h\mr{Gal}} $$
where the product ranges over the (finite, non-empty) set of mod $p$ points $(A,i,\lambda)$ of $\mc{X}_{V,L}$ with $F_{A[u^\infty]}$ of height $n$.
\end{proposition}

The groups $\aut(A,i,\lambda)$ are finite subgroups of the Morava stabilizer group.  The structure of these subgroups, and the conditions under which they are maximal finite subgroups, is studied in \cite{BehrensHopkins}.

\section{Further reading}\label{sec:further}

\begin{description}
\item{\bf Elliptic genera:} One of the original motivations behind $\tmf$ was Ochanine's definition of a genus of spin manifolds which takes values in the ring of modular forms for $\Gamma_0(2)$, which interpolates between the $\widehat{A}$-genus and the signature.  Witten gave an interpretation of this genus in terms of $2$-dimensional field theory, and produced a new genus (the \emph{Witten genus})\index{Witten genus} of \emph{string manifolds}\index{string manifold} valued in modular forms for $SL_2(\ZZ)$ \cite{WittenI}, \cite{WittenII}.  These genera were refined to an orientation of elliptic spectra by Ando-Hopkins-Strickland \cite{AHS}, and were shown to give $E_\infty$ orientations
\begin{align*}
MString & \rightarrow \tmf, \\
MSpin & \rightarrow \tmf_0(2)
\end{align*}
by Ando-Hopkins-Rezk \cite{AHR} and Wilson \cite{Wilson}, respectively.
 
\item{\bf Geometric models:} The most significant outstanding problem in the theory of topological modular forms is to give a geometric interpretation of this cohomology theory (analogous to the fact that $K$-theory classes are represented by vector bundles).  Motivated by the work of Witten described above, Segal proposed that $2$-dimensional field theories should represent $\tmf$-cocycles \cite{Segal}.  This idea has been fleshed out in detail by Stolz and Teichner, and concrete conjectures are proposed in \cite{StolzTeichner}. 

\item{\bf Computations of the homotopy groups of $\TMF_0(N)$:} Mahowald and Rezk computed the descent spectral sequence for $\pi_*\TMF_0(3)$ in \cite{MahowaldRezk}, and a similar computation of $\pi_*\TMF_0(5)$ was performed by Ormsby and the author in \cite{BehrensOrmsby} (see also \cite{HHR}).  Meier gave a general additive description of $\pi_*\TMF_0(N)^{\wedge}_2$ for all $N$ with $4 \not\vert \phi(N)$ in \cite{Meier}.  

\item{\bf Self-duality:} Stojanoska showed that Serre duality for the stack $\br{\mc{M}}_{ell}$ lifts to a self-duality result for $\Tmf[1/2]$.  This result was extended to $\Tmf_1(N)$ by Meier \cite{Meier}.

\item{\bf Detection of the divided $\beta$-family:} Adams used $K$-theory to define his $e$-invariant, and deduced that the order of the image of the $J$-homomorphism in degree $2k-1$ is given by the denominator of the Bernoulli number $\frac{B_k}{2k}$.  The divided $\beta$-family, a higher chromatic generalization of the image of $J$, was constructed on the $2$-line of the Adams-Novikov spectral sequence by Miller-Ravenel-Wilson \cite{MillerRavenelWilson}.  Laures used $\tmf$ to construct a generalization of the $e$-invariant, called the $f$-invariant \cite{Laures}.  This invariant relates the divided beta family to certain congruences between modular forms \cite{Behrenscong}, \cite{BehrensLaures}.  

\item{\bf Quasi-isogeny spectra:}  The author showed that the Goerss-Henn-Mahowald-Rezk resolution of the $3$-primary $K(2)$-local sphere  \cite{GHMR} can be given a modular interpretation in terms of isogenies of elliptic curves \cite{Behrens}, and conjectured that something similar happens at all primes \cite{Behrensbldg}.

\item{\bf The tmf resolution:} Generalizing his seminal work on ``$\bo$-resolutions,'' Mahowald initiated the study of the $\tmf$-based Adams spectral sequence.  This was used in \cite{BHHMI} to lift the $192$-periodicity in $\tmf_{(2)}$ to a periodicity in the $2$-primary stable homotopy groups of spheres, and in \cite{BHHMII} to show coker J is non-trivial in ``most'' dimensions less than 140.  The study of the $\tmf$-based Adams spectral sequence begins with an analysis of the ring of cooperations $\tmf_*\tmf$.  With $6$ inverted, this was studied by Baker and Laures \cite{Bakertmfcoop}, \cite{Laures}.  The $2$-primary structure of $\tmf_*\tmf$ was studied in \cite{BOSS}.

\item{\bf Dyer-Lashof operations:}  Ando observed that power operations for elliptic cohomology are closely related to isogenies of elliptic curves \cite{Ando}.  Following this thread, Rezk used the geometry of elliptic curves to compute the Dyer-Lashof algebra for the Morava $E$-theory $E_2$ at the prime $2$ \cite{RezkDL}.  This was generalized by Zhu to all primes \cite{ZhuDL}.  Using Rezk's ``modular isogeny complex'' \cite{RezkMIC}, Zhu was able to derive information about unstable homotopy groups of spheres \cite{Zhu}.

\item{\bf Spectral algebraic geometry:}  As mentioned in the introduction, Lurie introduced the notion of \emph{spectral algebraic geometry},\index{spectral algebraic geometry} and used it to give a revolutionary new construction of $\tmf$ \cite{Lurie} (see also \cite{LurieI} and \cite{LurieII}).

\item{\bf Equivariant TMF:} Grojnowski introduced the idea of complex analytic \emph{equivariant elliptic cohomology}\index{equivariant elliptic cohomology} \cite{Grojnowski}.  This idea was refined in the rational setting by Greenlees \cite{Greenlees}.  Lurie used his spectral algebro-geometric construction of $\TMF$ to construct equivariant $\TMF$ (this is outlined in \cite{Lurie}, see \cite{LurieI} and \cite{LurieII} for more details).

\item{\bf K3 cohomology:} Morava and Hopkins suggested that cohomology theories should be also be associated to K3 surfaces.  Szymik showed this can be done in \cite{Szymik}. \index{K3 cohomology}

\item{\bf Computations of $\pi_*\TAF$:} Very little is known about the homotopy groups of spectra of topological automorphic forms, for the simple reason that, unlike the modular case, very few computations of rings of classical integral automorphic forms exist in the literature.  Nevertheless, special instances have been computed in \cite{taf}, \cite{HillLawsonTAF}, \cite{BehrensLawson}, \cite{LawsonNaumann}, \cite{LawsonTAF}, \cite{vBT1}, \cite{vBT2}.
\end{description}
  
\bibliographystyle{amsalpha}
\nocite{*}
\bibliography{ch1}